\documentclass[12pt]{extarticle}
\usepackage{amsmath, amsthm, amssymb, color}
\usepackage[colorlinks=true,linkcolor=blue,urlcolor=blue]{hyperref}
\usepackage{graphicx}
\usepackage{caption}
\usepackage{mathtools}
\usepackage{enumerate}
\usepackage{verbatim}
\usepackage{tikz,tikz-cd,tikz-3dplot}
\usepackage{amssymb}
\usetikzlibrary{matrix}
\usetikzlibrary{arrows}
\usepackage{algorithm}
\usepackage[noend]{algpseudocode}
\usepackage{caption}
\usepackage[normalem]{ulem}
\usepackage{subcaption}
\oddsidemargin \evensidemargin
\usepackage{makecell}
\usepackage{array}
\usepackage[scr=rsfs]{mathalpha}
\usepackage{multirow,array}

\DeclareFontEncoding{LS1}{}{}
\DeclareFontSubstitution{LS1}{stix}{m}{n}
\DeclareSymbolFont{symbols2}{LS1}{stixfrak} {m} {n}
\DeclareMathSymbol{\operp}{\mathbin}{symbols2}{"A8}

\newtheorem{theorem}{Theorem}
\newtheorem{proposition}[theorem]{Proposition}

\newtheorem{corollary}[theorem]{Corollary}

\theoremstyle{definition}

\newtheorem{remark}[theorem]{Remark}

\newtheorem{example}[theorem]{Example}

\numberwithin{theorem}{section}

\newcommand{\bvec}[1]{\mathbf{#1}}

\newcommand{\vr}{\bvec{r}}

\newcommand{\vx}{\bvec{x}}

\newcommand{\vR}{\bvec{R}}

\newcommand{\PP}{\mathbb{P}}
\newcommand{\RR}{\mathbb{R}}

\newcommand{\CC}{\mathbb{C} }

\title{\bf Algebraic Varieties in Quantum Chemistry}

\author{Fabian M. Faulstich, Bernd Sturmfels and Svala Sverrisdóttir}

\date{}
\begin{document}
\maketitle

\begin{abstract} \noindent
  We develop algebraic geometry for
coupled cluster (CC) theory of quantum many-body systems.
  The high-dimensional eigenvalue problems that encode the electronic Schr\"odinger equation 
  are approximated by a hierarchy of polynomial systems at various levels of truncation.
  The exponential parametrization of the eigenstates
  gives rise to truncation varieties. These generalize
Grassmannians in their Pl\"ucker embedding.
We explain how to derive Hamiltonians,
we offer a detailed study of truncation varieties and their CC degrees,
and we present the state of the art in
solving the CC equations.
\end{abstract}
	
\section{Introduction}

Electronic structure theory is a powerful quantum mechanical framework 
for investigating the intricate behavior of electrons within molecules and crystals. At the core lies the interaction between particles, specifically the electron-electron and electron-nuclei interactions.
Embracing the essential quantum physical effects, this theory is the foundation for {\it ab initio} electronic structure calculations performed by many researchers in chemistry and related fields, complementing and supplementing painstaking laboratory work.
With its diverse applications in chemistry and materials science, electronic structure theory holds vast implications for 
the mathematical sciences.
Integrating methods from algebra and geometry
 into this field leads to the development of precise and scalable numerical methods, enabling extensive {\it in silico} studies of chemistry for e.g. sustainable energy, green catalysis, and nanomaterials. 
The synergy between  fundamental mathematics and electronic structure theory  
offers the potential for groundbreaking advancements in addressing these global challenges.

The electronic structure problem  is the innocent-looking eigenvalue problem in (\ref{eq:eigenvector}).
This has been under intense investigation since the 1920s, but it is still a formidable problem of contemporary science. 
The governing partial differential equation, i.e.~the electronic Schrödinger equation~(\ref{eq:schroedinger}), 
has $3d$ degrees of freedom, where $d$ is the number of electrons~\cite{Y}. 
As $d$ increases, the problem size of
  (\ref{eq:eigenvector}) grows exponentially.
Efficient and tractable numerical schemes are essential for approximating the behavior of complex atoms and molecules~\cite{LJ}. 
One class of widely employed high-accuracy methods rests on {\em coupled cluster theory}~\cite{HJO}, which is considered the gold standard of quantum chemistry.
Despite considerable successes, this approach has its limitations, leaving
many opportunities for further developments.

In this article, we develop algebraic geometry for coupled cluster theory, following \cite{FO}.
Central to our investigations are the {\em coupled cluster (CC) equations}. 
This system of polynomial equations had already been studied 
by the quantum chemistry community in the 1990s.
Researchers used homotopy continuation 
for numerical solutions, and they focused on  enhancing computational capabilities. 
 For details on this history
  we refer the interested reader to~\cite[Section 1-3]{KP1} and the references therein, or~\cite[Sections 1.1 and 1.2]{FO}.
   In contrast, our present study goes beyond advancements in computations, delving deeply into  
 the mathematical structures that underlie CC theory.
  It rests on
 nonlinear algebra \cite{BBCHLF, BHS, MS, Stu}.
 
On the geometry side, our point of departure is the Grassmannian ${\rm Gr}(d,n)$ with its $\binom{n}{d}$ Pl\"ucker coordinates. 
These represent quantum states for $d$ electrons in $n$ spin-orbitals.
We introduce a family of projective varieties $V_\sigma$ in the same ambient space $\PP^{\binom{n}{d}-1}$, one for each subset $\sigma$ of $[d]= \{1,2,\ldots,d\}$.
The singleton $\sigma = \{1\}$ yields the Grassmannian $ {\rm Gr}(d,n)$.

We formulate the coupled cluster equations as a truncated eigenvalue problem on  $V_\sigma$. 
The truncation varieties $V_\sigma$ for 
 $\sigma = \{1\}, \{2\}, \{1,2\},\{1,2,3\}$
correspond to the CC variants CCS, CCD, CCSD, CCSDT, explained in e.g.~\cite{FL, FO}.
The number of complex solutions to the CC equations for a general Hamiltonian $H$ is the {\em CC degree},
denoted ${\rm CCdeg}_{d,n}(\sigma)$.
This invariant reveals the number of paths that need to be tracked for finding all solutions.

Experts in nonlinear algebra will find
these concepts to be consistent with 
notions they are familiar with.
 The truncated eigenvalue problem on $V_\sigma$
is reminiscent of the theory of eigenvectors for tensors \cite[Section 9.1]{MS}.
The CC degree is an analog to the Euclidean distance (ED) degree \cite[Section 2]{Stu} 
and to the maximum likelihood (ML) degree \cite[Section~3]{Stu}.

The present article launches an entirely new line of research. 
By contrast, previous mathematical investigations of CC theory were performed within a functional analytic framework. In that framework, Schneider performed the first local analysis of CC theory in 2009 which was based on Zarantonello’s lemma~\cite{schneider2009analysis}. This approach was subsequently extended~\cite{rohwedder2013continuous,rohwedder2013error} and applied to different CC variants~\cite{laestadius2018analysis,laestadius2019coupled,faulstich2019analysis}.
Later Csirik and Laestadius established a more versatile framework for analyzing general CC variants which uses topological degree theory~\cite{csirik2023coupled1,csirik2023coupled2}.   
Recent numerical analysis results regarding single reference CC were established by Hassan, Maday, and Wang using the invertibility of the CC Fréchet derivative~\cite{hassan2023analysis}.

We now summarize the organization and contributions of this paper.
In Section \ref{sec2} we present the exponential parametrization which expresses the quantum states in terms of the
cluster amplitudes.
This map is invertible.
Theorem \ref{thm:UBP} gives formulas for all coordinates of the forward map and the backward map.
This involves a master polynomial of degree $d$ whenever $n=2d$.
For instance, for $d=2$, this polynomial is the Pl\"ucker quadric $\psi_{12} \psi_{34} - \psi_{13} \psi_{24} + \psi_{14} \psi_{23}$. In general, its monomials correspond to uniform block permutations~\cite{OSSZ}.

Section \ref{sec3} introduces the truncation variety $V_\sigma$ which lives
 in the projective space $\PP^{\binom{n}{d}-1}$.
Its restriction to an affine chart is 
  a complete intersection, defined by some coordinates of the backward map.
  The homogeneous prime ideal of $V_\sigma$ is found by saturation (Theorem~\ref{thm:primeideal}).
Theorem \ref{thm:grassmann} reveals the Grassmannian for $\sigma = \{1\}$.
Proposition \ref{prop:duality} features particle-hole symmetry $(d,n) \leftrightarrow (n{-}d,n)$, and
Theorem \ref{thm:linear} identifies all $\sigma$ for which $V_\sigma$
is a linear space.

Section \ref{sec4} starts from high school chemistry. 
It explains the discretization process in coupled cluster theory.
From the electronic Schr\"odinger equation we derive the Hamiltonian $H$, a symmetric matrix 
of size $\binom{n}{d} \times \binom{n}{d}$, which serves as the parameter in the CC equations.
The molecule
lithium hydride (LiH), with $d{=}4$ electrons in $n{=}8$ orbitals, is
our running~example. The Hamitonians derived in Example~\ref{ex:morerunning}
furnish the input data for Examples \ref{ex:ending48} and \ref{ex:ending38}.

Our theme in Section \ref{sec5} is the CC equations.
We define them in (\ref{eq:CC1})--(\ref{eq:trunceigen}) via a rank constraint  on $V_\sigma$, and we give a reformulation in terms of cluster amplitudes in (\ref{eq:formulation1})-(\ref{eq:formulation1b}).
 For most CC variants  used in computational chemistry,
this agrees with the traditional formulation~\cite{HJO}.
But our equations differ for some others (Theorem \ref{thm:twoformulations}).
Starting from the general bound in Theorem~\ref{thm:generalbound}, we offer a detailed study of the CC degrees of truncation~varieties.

In Section \ref{sec6} we turn to numerical solutions of the CC equations, both for generic Hamiltonians
and for systems derived from chemistry.
We present computations with the software {\tt HomotopyContinuation.jl} \cite{BT},
together with its certification feature \cite{BRT}.
The use of monodromy loops is essential.
Our findings show that the new theory leads to considerable practical advances.
This is documented in Examples \ref{ex:complexity} and \ref{ex:scaling}.
Examples \ref{ex:ending48} and \ref{ex:ending38}
offer case studies for lithium hydride (LiH) where $d=4,n=8$,
and for lithium (Li) where $d=3,n=8$.
For a comparison with previous work,
 \cite[Section~6]{FO} reports that CCSD with three electrons in six spin-orbitals ``supersedes the abilities of state-of-the-art algebraic geometry software'', and \cite[Theorem 4.10]{FO} offers the upper bound $2^{27} = 134217728$ for ${\rm CCdeg}_{3,6}(\{1,2\})$.
Yet, the true CC degree is $ 55$, by Proposition~\ref{prop:offby}. 
It is now instantaneous to solve these CC equations.
In short, new algebraic geometry leads to progress in practice.

\section{Exponential Parametrization}
\label{sec2}

We work in the vector space $\mathcal{H} = \wedge_d \RR^n$
with its standard basis vectors $e_I = e_{i_1} \wedge e_{i_2} \wedge  \cdots \wedge e_{i_d}$.
In this notation, $I=(i_1< i_2 <\ldots<i_d) \in \binom{[n]}{d}$ is a subset of $[n]$ of cardinality $d$ whose elements are always written in (increasing) order.
Here $d \leq n$ are positive integers.
The {\em reference state} is the first basis vector $e_{[d]}$ for $[d] = \{1,2,\ldots,d\}$.
Vectors in $\mathcal{H}$ are called {\em quantum states} and they are written uniquely 
as linear combinations of the basis vectors:
$$ \psi \,\,= \sum_{I \in \binom{[n]}{d}} \psi_I  \, e_I . $$
Motivated by nonlinear algebra \cite[Chapter 5]{MS},
we call $\psi_I$ the {\em Pl\"ucker coordinates}.
Sometimes it is preferable to write the Pl\"ucker coordinates as
$c_{\alpha,\beta}$, where $\alpha $ is a subset of $[d]$ and $\beta $ is a subset of $[n] \backslash [d] = \{d+1,\ldots,n\}$ of the same cardinality $|\alpha| = |\beta|$.
The $c_{\alpha,\beta}$ are known as {\em configuration interaction coefficients}
in quantum chemistry.
The identification between these two systems of coordinates on 
the space of quantum states $\mathcal{H}
= \wedge_d \RR^n$ is  as follows:
 \begin{equation}
 \label{eq:cpsi} \qquad c_{\alpha,\beta} \,=\, \psi_I
 \qquad {\rm where} \quad I = ([d] \backslash \alpha) \cup \beta. 
 \end{equation}
We think of the $\psi_I$ as the $d \times d$ minors of a $d \times n$ matrix,
and we think of the $c_{\alpha,\beta}$ as the minors of all sizes
in a $d \times (n-d)$ matrix. These two sets have the same cardinality 
$$ \binom{n}{d} \,\,\,\, = \, \sum_{k=0}^{{\rm min}(d,n-d)}\binom{d}{k} \binom{n-d}{k}. $$

The title of this section refers to a birational map, defined shortly, between two copies of $\RR^{\binom{n}{d}}$.
It restricts to a polynomial map with polynomial inverse on the affine hyperplane
$$ \mathcal{H}' \,\,\,= \,\,\, \{ \psi \in \mathcal{H} \,:\, \psi_{[d]} = 1 \} \,\,\,
\simeq \,\,\,  \RR^{\binom{n}{d}-1}.$$
Later on, when we come to algebraic varieties, we shall pass from the
vector space $\mathcal{H}$ to the projective space $\PP^{\binom{n}{d}-1} = \PP(\mathcal{H})$.
This is the projective closure of the
affine space $\mathcal{H'} $.
Thus the above coordinates $\psi_I$ and $c_{\alpha,\beta}$ also serve as homogeneous coordinates on
$\PP^{\binom{n}{d}-1}$.
See \cite[Chapter 8]{CLO}
for basics on projective algebraic geometry 
with a view toward computation.

To define the exponential parametrization, we introduce our second
vector space $\mathcal{V}$.
This is isomorphic to $\mathcal{H}$, with
coordinates indexed by $\binom{[n]}{d}$.
The elements of $\mathcal{V}$ are called {\em cluster amplitudes}
and we denote them by
$x = ( x_I )_{I \in \binom{[n]}{d}}$.
The cluster amplitudes $x$ also have
 alternate coordinates that are indexed by minors
of a $d \times (n-d)$-matrix. As in (\ref{eq:cpsi}), we~set
\begin{equation}
 \label{eq:tx} 
 \qquad t_{\alpha,\beta} \,=\, x_I
 \qquad {\rm where} \quad I = ([d] \backslash \alpha) \cup \beta. 
 \end{equation}

The {\em level} of a coordinate $\psi_I$ or $x_I$ is
defined as the cardinality of $I \backslash [d]$. 
Equivalently, the level of $c_{\alpha,\beta}$ or $t_{\alpha,\beta}$  equals
$|\alpha| = |\beta|$. For example,
for $d=3$ and $n=6$, each of the spaces 
$\mathcal{H}$ and $\mathcal{V}$ has $20$ coordinates:
one of level $0$,
nine of level $1$,
nine of level $2$,
and one of level $3$:
$$ \begin{matrix}
\psi_{123} = c_\emptyset,
 \psi_{124} = c_{3,4}, 
 \psi_{125} = c_{3,5}, \ldots,
 \psi_{136} = c_{2,6},
 \psi_{145} = c_{23,45}, \ldots,
\psi_{356} = c_{12,56}, 
\psi_{456} = c_{123,456}, \\
x_{123} = t_\emptyset,\,
 x_{124} = t_{3,4}, 
 x_{125} = t_{3,5}, \ldots,\,
 x_{136} = t_{2,6},\,
 x_{145} = t_{23,45}, \ldots,
x_{356} = t_{12,56},\, 
x_{456} = t_{123,456}.
\end{matrix}
 $$
The term level refers to the
excitation level of the electrons in a chemical system.

Our workhorse is the nonlinear coordinate transformation between 
quantum states and cluster amplitudes. The basic ingredient is 
a lower-triangular matrix $T(x)$ of square format
$\binom{n}{d} \times \binom{n}{d}$.
The entry of $T(x)$ in row $J$ and column $I$ is zero
unless $I \backslash J \subseteq [d]$ and
$(J \backslash I) \cap [d] = \emptyset$. If this holds then
the matrix entry is $\,\pm \,t_{I \backslash J ,\, J \backslash I}$,
where the sign is defined as follows.
Similarly to $I$, the sets $J, \, I \backslash J, \,J \backslash I$ and $I \cap J$
are subsets whose elements are written in (increasing) order. To be precise,
if $I = (i_1< i_2 < \cdots<i_d)$,
then the sequence $(I \cap J, I \backslash J)$ is a permutation of $I$.
The sign of $t_{I\backslash J, J\, \backslash I}$ in $T(x)_{J,I}$
is the sign of the permutation $I \mapsto (I \cap J, I  \backslash J)$
times the sign of the permutation $J \mapsto (I \cap J, J  \backslash I)$.

Using $x$-coordinates, the entry of $T(x)$ in row $J$ and column $I$
equals the above sign times
$$\qquad \qquad x_K \,\, = \,\, t_{I \backslash J ,\, J \backslash I},
\qquad {\rm where} \quad K = [d] \backslash (I \backslash J) \cup (J \backslash I).
$$
In conclusion,  $T(x)$ is a well-defined lower-triangular matrix
of size $\binom{n}{d} \times \binom{n}{d}$
that depends linearly on the cluster amplitudes $x$.
This matrix represents the {\em cluster operator} in \cite{FL, FO}.

\begin{example}[$d=2,n=5$] \label{ex:zweifuenf1}
The lower-triangular $10 \times 10$ matrix defined above equals
$$  T(x) \,=\,
\begin{footnotesize}
\begin{bmatrix}
 0 &    0 &    0 &    0 &   0 &    0 &    0 &  0 & 0 \,& 0 \,\,\, \\
 x_{13} &   0 &    0 &    0 &   0 &    0 &    0 &  0 & 0 \,& 0 \,\,\, \\
 x_{14} &   0 &    0 &    0 &   0 &    0 &    0 &  0 & 0 \,& 0 \,\,\, \\
 x_{15} &   0 &    0 &    0 &   0 &    0 &    0 &  0 & 0 \, & 0 \,\,\, \\
 -x_{23} &  0 &    0 &    0 &   0 &    0 &    0 &  0 & 0 \, & 0 \,\,\, \\
 -x_{24} &  0 &    0 &    0 &   0 &    0 &    0 &  0 & 0 \, & 0 \,\,\, \\
  -x_{25} &  0 &    0 &    0 &   0 &    0 &    0 &  0 & 0 \, & 0 \,\,\, \\
   x_{34} &  -x_{24} & x_{23} &   0 &  -x_{14} & x_{13} &   0 &  0 & 0 \, & 0 \,\,\, \\
   x_{35} &  -x_{25} &  0 &   x_{23} & -x_{15} &  0 &   x_{13} & 0 & 0 \, & 0 \, \,\,\\
   x_{45} &   0 & \!\! -x_{25} & x_{24} &  0 &  \!\!  -x_{15} & x_{14} & 0 & 0 \, & 0 \,\,\,
   \end{bmatrix}.
\end{footnotesize}
$$                     
The level zero variable $x_{12}$ does not appear. Three variables 
$x_{34}, x_{35}, x_{45}$ have level two.
\end{example}

The lower-triangular matrix $T(x)$ is  nilpotent of order $d$.
This is shown in \cite[Section 3.2]{FO} and also follows from equation (\ref{eq:T(t)}).
Hence, the matrix exponential is the finite~sum
\begin{equation}
\label{eq:expT}
 {\rm exp}(T(x)) \, = \,\, \sum_{k=0}^d \frac{1}{k!}\, T(x)^k. 
 \end{equation}
In particular, the entries of the matrix ${\rm exp}(T(x))$
 are polynomials in $x$ of degree at most $d$.

\begin{example}[$d=2,n=5$] \label{ex:zweifuenf2}
The exponential of the matrix in Example \ref{ex:zweifuenf1} equals
\begin{equation}
\label{eq:tenbyten}  {\rm exp}(T(x)) \,=\,
\begin{footnotesize}
\begin{bmatrix}
 1 &    0 &    0 &    0 &   0 &    0 &    0 &  0 & 0 \,& 0 \,\,\, \\
 x_{13} &   1 &    0 &    0 &   0 &    0 &    0 &  0 & 0 \,& 0 \,\,\, \\
 x_{14} &   0 &    1 &    0 &   0 &    0 &    0 &  0 & 0 \,& 0 \,\,\, \\
 x_{15} &   0 &    0 &    1 &   0 &    0 &    0 &  0 & 0 \, & 0 \,\,\, \\
 -x_{23} &  0 &    0 &    0 &   1 &    0 &    0 &  0 & 0 \, & 0 \,\,\, \\
 -x_{24} &  0 &    0 &    0 &   0 &    1 &    0 &  0 & 0 \, & 0 \,\,\, \\
  -x_{25} &  0 &    0 &    0 &   0 &    0 &    1 &  0 & 0 \, & 0 \,\,\, \\
 x_{14} x_{23} {-} x_{13} x_{24} +   x_{34} &  -x_{24} & x_{23} &   0 &  -x_{14} & x_{13} &   0 &  1 & 0 \, & 0 \,\,\, \\
 x_{15} x_{23} {-} x_{13} x_{25} + x_{35} &  -x_{25} &  0 &   x_{23} & -x_{15} &  0 &   x_{13} & 0 & 1 \, & 0 \, \,\,\\
 x_{15} x_{24} {-} x_{14} x_{25} +  x_{45} &   0 & \!\! -x_{25} & x_{24} &  0 &  \!\!  -x_{15} & x_{14} & 0 & 0 \, & 1 \,\,\,
   \end{bmatrix}.
\end{footnotesize}
\end{equation}
The following observation will be important later on:
If we set the level two parameters to zero, i.e.~$x_{34} = x_{35} = x_{45} = 0$,
then the first column consists of all $10$ maximal minors of 
$$
\begin{bmatrix}
\,\,\,1 & \,\,0\, &  x_{23} & x_{24} & x_{25} \, \,\\
\,\,\,0 & \,\,1\, &  x_{13} & x_{14} & x_{15} \,\,
\end{bmatrix} \quad = \quad
 \begin{bmatrix}
\,\,\,1 & \,\,0\, & t_{1,3} & t_{1,4} & t_{1,5} \, \,\\
\,\,\,0 & \,\,1\, & t_{2,3} & t_{2,4} & t_{2,5} \,\,
\end{bmatrix}. $$
Thus the Grassmannian ${\rm Gr}(2,5) \subset \PP^9$ makes an appearance in
the first column of ${\rm exp}(T(x))$.
\end{example}

Returning to general $d$ and $n$, the {\em exponential parametrization} is the map
\begin{equation} \label{eq:xtopsi} \qquad
\mathcal{V} \,\rightarrow \,\mathcal{H} , \, x \,\mapsto \,\psi \, , \qquad {\rm where} \quad
 \psi \, = \, {\rm exp}(T(x))   \, e_{[d]} . \end{equation}
 Here $e_{[d]}$ is the reference state in $\mathcal{H} \simeq \RR^{\binom{n}{d}}$, i.e.~the first
 standard basis vector.
 The transformation (\ref{eq:xtopsi}) gives a formula for
 the quantum states $\psi$ in terms of the
cluster amplitudes $x$. To be precise, each of the
$\binom{n}{d}$ coordinates $\psi_I$ is 
 a polynomial $\psi_I(x)$ in the $\binom{n}{d}$ unknowns $x_J$.
 
 In the definition (\ref{eq:xtopsi}),  we had assumed that $x_{[d]} = 1$ and  $\psi_{[d]} = 1$.
 Geometrically, this means that we
work in the affine spaces $\mathcal{V}'$ and $\mathcal{H}'$, both of which are identified with
$\RR^{\binom{n}{d}-1}$.
Later on, we shall extend  (\ref{eq:xtopsi}) to a birational automorphism of the
 projective space  $\PP^{\binom{n}{d}-1}$.
The reference coordinates $x_{[d]} = t_{\emptyset,\emptyset}$
and $\psi_{[d]}= c_{\emptyset, \emptyset} $ will then
serve as homogenizing variables.

The formula $\, \psi  =  {\rm exp}(T(x))   e_{[d]} \,$ simply says that
$\psi$ equals the leftmost column vector of the matrix ${\rm exp}(T(x))$.
For instance,  in Example \ref{ex:zweifuenf2},
the formula for the Pl\"ucker coordinates of the
quantum state $\psi = (\psi_{12},\psi_{13},\ldots,\psi_{45})$
in terms of cluster amplitudes is given by
 the first column  of~(\ref{eq:tenbyten}).
Here is a slightly larger example, where the matrices have size $20 \times 20$:

\begin{example}[$d=3,n=6$] \label{ex:dreisechs}
The $20$ coordinates in the formula (\ref{eq:xtopsi}) are as follows:
$$
\begin{footnotesize}
\begin{aligned}
\psi_{123} &= x_{123} = 1   &\psi_{135} &= -x_{135}    &\psi_{145} &= x_{145} - x_{124} x_{135}+x_{125} x_{134}        &\psi_{256} &= - x_{256} +  x_{125} x_{236} - x_{126} x_{235} \\[-0.5em]
\psi_{124} &= x_{124}       &\psi_{136} &= -x_{136}     &\psi_{146} &= x_{146} - x_{124} x_{136}+x_{126} x_{134}        &\psi_{345} &= x_{345} - x_{134} x_{235} + x_{135} x_{234} \\[-0.5em]
\psi_{125} &= x_{125}       &\psi_{234} &= x_{234}      &\psi_{156} &= x_{156}- x_{125} x_{136} + x_{126} x_{135}       &\psi_{346} &= x_{346} -x_{134} x_{236}+x_{136} x_{234} \\[-0.5em]
\psi_{126} &= x_{126}       &\psi_{235} &= x_{235}      &\psi_{245} &=  -x_{245} +  x_{124} x_{235} - x_{125} x_{234}   &\psi_{356} &= x_{356} -x_{135} x_{236}+x_{136} x_{235}\\[-0.5em]
\psi_{134} &= -x_{134}      &\psi_{236} &= x_{236}      &\psi_{246} &= - x_{246} + x_{124} x_{236} - x_{126} x_{234}    \\[-0.5em]
\end{aligned}
\end{footnotesize}
$$
Finally, at level three we find that $\psi_{456}$ is equal to
$$  \begin{footnotesize} \begin{matrix}
x_{456} \, + \,
x_{124} x_{356} -  x_{125} x_{346} +x_{126} x_{345} -x_{134} x_{256}
+x_{135} x_{246} - x_{136}x_{245} +x_{145} x_{236} -x_{146} x_{235} +x_{156} x_{234} \\ 
\,-\,\,x_{124} x_{135} x_{236}+x_{124} x_{136} x_{235}+x_{125} x_{134} x_{236}
-x_{125} x_{136} x_{234}-x_{126} x_{134} x_{235}  + x_{126} x_{135} x_{234}.
\end{matrix}
\end{footnotesize}
$$
\end{example}

In both examples, we can easily solve
the equation $ \,\psi  =  {\rm exp}(T(x)) \, e_{[d]} \,$  for $x$.
This is done inductively by  level.
For levels zero and one, we simply have $x_I = \pm \psi_I$.
At each larger level, we use the formulas for
lower level coordinates $x_I$ in terms of the $\psi_J$,
and we substitute these into the equation. For instance, in 
Example \ref{ex:dreisechs}, this yields the inversion formulas 
$$  
\begin{footnotesize} 
\begin{aligned}
x_{124} &= \psi_{124} ,\; \ldots \; ,\, \psi_{236} = x_{236}\\
x_{145} &= \psi_{145}-\psi_{124} \psi_{135} + \psi_{125} \psi_{134}, \;\ldots \; ,
\, \psi_{356} = x_{356} -x_{135} x_{236}+x_{136} x_{235}\\
x_{456} &= \psi_{456} -\psi_{124} \psi_{356}+\psi_{125} \psi_{346}-\psi_{126} \psi_{345}+\psi_{134} \psi_{256}
-\psi_{135} \psi_{246}+\psi_{136} \psi_{245}-\psi_{145} \psi_{236}+\psi_{146} \psi_{235}\\ 
&- \psi_{156} \psi_{234}
+ 2 (\psi_{124} \psi_{135} \psi_{236} {-}  \psi_{124} \psi_{136} \psi_{235} {-}  \psi_{125} \psi_{134} \psi_{236}
{+} \psi_{125} \psi_{136} \psi_{234} {+}  \psi_{126} \psi_{134} \psi_{235} {-} \psi_{126} \psi_{135} \psi_{234}).
\end{aligned}
\end{footnotesize}
$$

We next show that such inversion formulas can always be found.

\begin{proposition} \label{prop:invertmap}
The map (\ref{eq:xtopsi}) from 
$x$-coordinates to $\psi$-coordinates  has a polynomial inverse.
Namely,
 $x_I$ is equal to $\pm \psi_I$ plus a polynomial
in $\psi$-coordinates of strictly lower level.
\end{proposition}

\begin{proof}
We proceed by induction on the level. At level zero, we have $x_{[d]} = \psi_{[d]} = 1$.
For level one, we already saw that $x_I = \pm \psi_I$.     If the index $I$ has level $r$ 
then the formula in (\ref{eq:xtopsi}) writes $\psi_I$ as $\pm x_I$ plus a polynomial in variables $x_J$ of level $< r$. 
Each of these lower level $x_J$ can now be replaced with a polynomial in $\psi$, by the induction hypothesis. This yields the promised representation for $x_I$ as $\pm \psi_I$ plus a polynomial in lower level $\psi$-coordinates.
  \end{proof}

Let $x_I(\psi)$ denote the polynomial that expresses the cluster amplitudes 
in terms of the Pl\"ucker coordinates. Conversely,
 $\psi_I(x)$ is a polynomial in cluster amplitudes. 
As is apparent from the proof of Proposition~\ref{prop:invertmap}, the degree of each polynomial is the level of $I$ and the variable $x_I$ occurs linearly in $\psi_I(x)$. Similarly, the variable $\psi_I$ occurs linearly in~$x_I(\psi)$.  

The monomials occurring in $\psi_I(x)$ are
in  natural bijection with those occurring in $x_I(\psi)$, and
we shall give an explicit formula for these monomials and their coefficients.
For this, it helps to note that, for each degree $d$, there is really only one
$\psi$-polynomial and $x$-polynomial.
Here we refer to the symmetric group actions that
permute the indices in $[d]$ and in $[n] \backslash [d]$.
For fixed $d$ and fixed level $| I \backslash [d] |$, all 
$\psi_I(x)$ are in the same orbit, and ditto for~$x_I(\psi)$.  
Each coordinate in the exponential parametrization is a replicate of a certain {\it master polynomial}.

It would be desirable to better understand our equations
from the perspective of representation theory. In this setting,
the master polynomials should be highest weight vectors.
 
 \smallskip
 
The master polynomials of degree $d$ are $\psi_I(x)$ and $x_I(\psi)$ where $n=2d$ and $I = [2d]\backslash [d]$.
All coordinates in (\ref{eq:xtopsi}) and its inverse are obtained from these two by changing indices.
For instance, all quadratic entries in the matrix (\ref{eq:tenbyten}) are replicates of 
$\psi_{34}(x) = x_{14} x_{23} - x_{13} x_{24} + x_{34}$.

By inverting the map (\ref{eq:xtopsi}), we find the  following master polynomials of degree $d \leq 5$:
$$ \begin{footnotesize} \begin{matrix}
x_{34}(\psi) &=& \psi_{34} - \psi_{13} \psi_{24} + \psi_{14} \psi_{23} && \hbox{3 terms} \\
x_{456}(\psi) &=& \psi_{456} - \psi_{124} \psi_{356} +  \cdots - 2 \psi_{126} \psi_{135} \psi_{234}  && \hbox{16 terms} \\
x_{5678}(\psi)  &= & \psi_{5678} - \psi_{1235} \psi_{4678}  +  \cdots -2 \psi_{1278} \psi_{1346} \psi_{2345} +  
\cdots -6 \psi_{1238} \psi_{1247} \psi_{1346} \psi_{2345}
&& \hbox{131 terms} \\
x_{67890}(\psi) & = &  \psi_{67890} - \psi_{12346} \psi_{57890} +  \cdots 
-2  \psi_{12890} \psi_{13457} \psi_{23456} +  \cdots \qquad \qquad \qquad \qquad \quad \, && \\
& & +\,6 \psi_{12390} \psi_{12458} \psi_{13457} \psi_{23456} +  \cdots
+24 \psi_{12349} \psi_{12358} \psi_{12457} \psi_{13456} \psi_{23450}
& & \hbox{1496 terms}
\end{matrix}
\end{footnotesize}
$$
The first Pl\"ucker coordinate $\psi_{[d]}$ plays a special role. It
does not occur in our polynomials.
Let $\bar x_I(\psi)$ denote the homogenization of
$x_I(\psi)$ with respect to $\psi_{[d]}$. This
is a homogeneous polynomial of
degree $| I \backslash [d] |$. 
Its hypersurface $V(\bar x_I) \subset \PP^{\binom{n}{d}-1}$ will be important in Section~\ref{sec3}.

We close this section by giving explicit combinatorial formulas for the master polynomials. 
Formulas for all other coordinates in (\ref{eq:xtopsi}) and their inverse are obtained by adjusting the indices.
The number of monomials is found in {\em The On-Line Encyclopedia of Integer Sequences}, which is
published electronically at \url{http://oeis.org}. Namely, it is the sequence
$$ \qquad \# \,\mathcal{U}_d \,=\, 3, 16, 131, 1496, 22482, 426833, 9934563
\quad {\rm for} \,\,\,
d = 2,3,4,5,6,7,8. \qquad 
 (A023998)     
$$
This is  the number of uniform block permutations. Let
$[\overline{d}] = [2d]\backslash [d]$. We recall (e.g.~from \cite{OSSZ}) that
a {\em uniform block permutation} is a partition $\pi$ of the set $[d] \cup [\overline{d}] = [2d]$, 
here denoted
\begin{equation}
\label{eq:piUBP} \pi\,=\,
\{\pi_1,\pi_2,\ldots,\pi_k \} \,=\,
\{\alpha_1 \cup \beta_1, \alpha_2 \cup \beta_2, \ldots, \alpha_k \cup \beta_k\},
\end{equation}
which satisfies $\alpha_i\subseteq[d]$, $\beta_i \subseteq [\overline{d}]$ and $|\alpha_i| = |\beta_i|$ for all $i\in [k]$.
We denote by $\mathcal{U}_d$ the set of all uniform block permutations of $[2d]$.
For details on the algebraic and combinatorial structures of $\mathcal{U}_d$ 
we refer to \cite[Section 2.3]{OSSZ}
and the references therein.

The connection to coupled cluster theory arises from identifying $\mathcal{U}_d$ with the set of monomials 
that appear in the above polynomials $\psi_I(x)$ and $x_I(\psi)$.
To see this, we pass to the coordinates in (\ref{eq:cpsi})  and (\ref{eq:tx}).
The polynomial $\psi_I(x)$ is now written as $c_{\alpha,\beta}(t)$, and $x_I(\psi)$ is now written as  $t_{\alpha,\beta}(c)$.
With this notation, the master polynomials of degree $d$ are $c_{[d],[\overline d]}(t)$ and $t_{[d],[\overline d]}(c)$, and we write the monomial corresponding to a 
given uniform block permutation as
\begin{equation}
\label{eq:tcmono}
t_\pi \,:= \, t_{\alpha_1,\beta_1}  t_{\alpha_2,\beta_2}  \cdots  \, t_{\alpha_k,\beta_k}
\quad {\rm and} \quad   
c_\pi \,:=\, c_{\alpha_1,\beta_1}  c_{\alpha_2,\beta_2}  \cdots \, c_{\alpha_k,\beta_k} .   
\end{equation} 
The master polynomials are $\mathbb{Z}$-linear combinations of these monomials for $k=1,2,\ldots,d$.
 
\begin{theorem} \label{thm:UBP}
The coordinates in the exponential parametrization are 
$$ 
c_{[d],[\overline{d}]}(t) \,=\, 
\sum_{\pi \in\mathcal{U}_d}{\rm sign}(\pi) \,t_\pi 
\quad {\rm and} \quad
t_{[d],[\overline{d}]}(c) \,=  \,
\sum_{\pi \in \mathcal{U}_d}  {\rm sign}(\pi) (-1)^{\nu+ k - 1} (k-1)! \,c_\pi.
$$
In these formulas, the sign of an element $\pi \in \mathcal{U}_d$ is the product of the signs of the  permutations
\begin{equation}
\label{eq:signperm}
[d] \mapsto (\alpha_1, \alpha_2, \dots, \alpha_k)
 \quad {\rm and} \quad [\overline{d}] \mapsto (\beta_1, \beta_2, \dots, \beta_k),
\end{equation}
where each $\alpha_i$ and $\beta_i$ is an increasing sequence.
We also have
 $\nu = \frac{d(d - 1)}{2} - \sum_{r = 1}^k\! \frac{|\alpha_r|(|\alpha_r| - 1)}{2}$. 
\end{theorem}

\begin{proof}[Proof]
Let $\pi \in \mathcal{U}_d$.
We consider the monomial $t_\pi$ in the master polynomial $c_{[d],[\overline{d}]}(t)$. 
Its term is the product of matrix entries of $T(x)$ 
whose respective rows and columns are
$$
([d] \backslash A_{r} ) \cup B_{r} \quad {\rm and} \quad ([d] \backslash A_{r - 1} ) \cup B_{r - 1},
$$
where $A_r = \cup_{i=1}^r \alpha_i$ and $B_r = \cup_{i=1}^r \beta_i$.
One checks from the definition of $T(x)$ that the sign of such an entry comes from the number of inversions of $(B_{r - 1}, \beta_{r})$ and $(\alpha_{r}, [d] \backslash A_{r})$ and from the sign $(-1)^{|\alpha_r|(d - |\alpha_r|)}$.
We take the product of the signs of these entries for all $r \le k$. 
Since
$$
\sum_{r = 1}^k |\alpha_r|(d - |\alpha_r|) \,=\, \sum_{r = 1}^k |\alpha_r|(d - 1) - \sum_{r = 1}^k |\alpha_r|(|\alpha_r| - 1)
\, =\, d(d - 1) - \sum_{r = 1}^k |\alpha_r|(|\alpha_r| - 1)
$$
is an even integer, 
the sign of $\pi$ equals the product of the signs of the permutations in 
(\ref{eq:signperm}).

The set $\mathcal{U}_d$ of uniform block permutations has a natural partial order, induced by the partial orders on
the set partitions of $[d]$ and $[\overline{d}]$.
The Möbius function of the poset $\mathcal{U}_d$ is given by $\mu(\pi) = (-1)^{k - 1}(k - 1)!$.
For any uniform block partition $\rho \in \mathcal{U}_d$ we can write
$$
(-1)^\nu   {\rm sign}(\rho) \,  c_\rho(t) \,\,=\,\, \sum_{\pi \le \rho} {\rm sign} (\pi)  \, t_\pi .
$$
Using Möbius inversion, we obtain the asserted formula for  $t_{[d], [\overline{d}]}$ in terms of
the $c_\pi$.
\end{proof}

\section{Truncation Varieties}
\label{sec3}

In this section, we study the algebraic varieties that are promised in the title of this article.
They are found by truncating the exponential parametrization (\ref{eq:xtopsi}) to a certain coordinate subspace.
We consider the image of this truncation in $\mathcal{H}$.
Its closure in $\PP^{\binom{n}{d}-1}$ is our variety.

More precisely, let $\sigma$  be a non-empty proper subset of $[d] = \{1,2,\ldots,d\}$, and define
\begin{equation}
\label{eq:Vsigma} \begin{matrix}
\mathcal{V}_\sigma
\,\,=\,\,
{\rm span} \,\bigl\{  \,e_J\,:\,
J \in \binom{[n]}{d} \,\,{\rm and} \,\,
|J \backslash [d]| \in \sigma \cup \{0\} \bigr\}. 
\end{matrix} \end{equation}
This is a linear subspace of  the vector space $\mathcal{V}$ spanned by all basis vectors $e_J$ of level in $\sigma$.
 The subspace $\mathcal{V}_\sigma$ is the variety of the ideal
 \begin{equation} 
 \label{eq:Pideal} \begin{matrix}
 P_\sigma \,\, = \,\,
 \bigl\langle
 \, x_I \,\,: \,\, I \in \binom{[n]}{d} \,\,{\rm and} \,\,
| I \backslash [d]| \in [d] \backslash \sigma  \,\bigr\rangle .
\end{matrix}
\end{equation}
  
The restriction of  the exponential parametrization to  the subspace 
$\mathcal{V}_\sigma$ is injective. It maps $\mathcal{V}_\sigma$ into 
the full space of quantum states $\mathcal{H}$, and it maps further to
the projective space $\,\PP(\mathcal{H}) = \PP^{\binom{n}{d}-1}$.
We define the {\em truncation variety} $V_\sigma$ as the closure of the
image of $\mathcal{V}_\sigma$ under this map to $\PP^{\binom{n}{d}-1}$. Since the exponential parametrization is invertible, 
the dimension of the projective variety $V_\sigma$ is one less than 
the dimension of its linear space of parameters
$\mathcal{V}_\sigma$.

The varieties $V_\sigma$ correspond to the
various models in  CC theory. For instance,
in the notation of \cite[Section 1.2]{FO}, the index set $\sigma = \{1,2\}$
corresponds to CCSD, the index set $\sigma = \{1,2,3\}$  corresponds to CCSDT, etc.
But here we allow arbitrary truncation sets. For instance, taking $ \sigma = \{2,3\}$
means that doubles and triples are included but singles are not.

\begin{example}[$d{=}2, n{=}5$]
There are only two proper subsets of $[d]$, namely
$\sigma = \{1\}$ and $\sigma = \{2\}$.
The varieties $V_\sigma$ live in $\PP^9$. They
are defined by truncating the exponential parametrization, which is
given by the leftmost column in (\ref{eq:tenbyten}).
For $\sigma = \{2\}$ we set $x_{13} = x_{14} = x_{15} 
= x_{23} = x_{24} = x_{25} = 0$. Hence
$V_{2}$ is the subspace $\PP^3$
with coordinates $(\psi_{12}: \psi_{34}:\psi_{35}:\psi_{45})$.
For $\sigma = \{1\}$ we set $x_{34} = x_{35} = x_{45} = 0$,
and we obtain the Grassmannian ${\rm Gr}(2,5) $.
See Example \ref{ex:zweifuenf2} and
Theorem \ref{thm:grassmann}. We revisit the ideal of ${\rm Gr}(2,5)$ in
Example~\ref{ex:zweifuenf3}.
\end{example}

Our next result characterizes the homogeneous prime ideals of the truncation varieties.
It shows how to derive these ideals from the master polynomials that are given in Theorem~\ref{thm:UBP}.

\begin{theorem} \label{thm:primeideal}
The homogeneous prime ideal of the truncation variety $V_\sigma$ is  the saturation
\begin{equation}
\label{eq:saturationideal}
\mathcal{I}(V_\sigma) \,\,\,= \,\,\,
 \bigl\langle\, \bar{x}_I(\psi) \, : \, | I \backslash [d] | \in [d] \backslash \sigma  \bigr\rangle \,: \,
\langle\, \psi_{[d]} \,\rangle^\infty .
\end{equation}
In particular,  the restriction of $\,V_\sigma$ to the
affine chart $\,\CC^{\binom{n}{d}-1} = \{ \psi_{[d]} =1 \}$ of 
projective space $\,\PP^{\binom{n}{d}-1}$
is the complete intersection defined by the equations
$x_I(\psi) = 0$ where $|I \setminus [d]| \in [d] \backslash \sigma$.
\end{theorem}

The provided explicit description of the ideal of the truncation variety allows the computation of ${\rm deg}(V_\sigma)$ -- and hence the bound in  Theorem \ref{thm:generalbound} --
via the degree of the ideal.
For a textbook introduction to the saturation in (\ref{eq:saturationideal}), see Definition 8 in \cite[Section 4.4]{CLO}.
Its role in the context of projective geometry is explained in \cite[Section 8.5]{CLO}.
These references and the following example are meant to help our readers in understanding Theorem~\ref{thm:primeideal}.

\begin{example}[$d{=}2, n{=}5$, $\sigma {=} \{1\}$]
 \label{ex:zweifuenf3}
 The truncation variety $V_{\{1\}} = {\rm Gr}(2,5)$ has
codimension $3$  in $\PP^9$. Its restriction to the
affine chart $\CC^9 = \PP^9 \backslash V(\psi_{12})$ is the zero set of three polynomials:
$$ \begin{matrix}
x_{34}(\psi) & = & \, \psi_{34} - \psi_{13} \psi_{24} + \psi_{14} \psi_{23}, \\
x_{35}(\psi) & = & \, \psi_{35} - \psi_{13} \psi_{25} + \psi_{15} \psi_{23}, \\
x_{45}(\psi) & = & \, \psi_{45} - \psi_{14} \psi_{25} + \psi_{15} \psi_{24}.
\end{matrix}
$$
By multiplying each first term with $\psi_{12}$, we obtain the quadratic forms
 $\bar{x}_{34}(\psi), \bar{x}_{35}(\psi),\bar{x}_{45}(\psi)$.
These do not cut out $V_{\{1\}}$. Indeed,
the ideal on the left below is radical but it is not prime:
$$
    \! \langle \bar{x}_{34}(\psi), \bar{x}_{35}(\psi),\bar{x}_{45}(\psi) \rangle  = \,
\mathcal{I}({\rm Gr}(2,5)) \,\cap \,
\langle \psi_{12} ,  \psi_{13} \psi_{24} - \psi_{14} \psi_{23},
                          \psi_{13} \psi_{25} - \psi_{15} \psi_{23},
                          \psi_{14} \psi_{25} - \psi_{15} \psi_{24} \rangle.
$$
    This is a complete intersection, 
of codimension $3$ and degree $2^3 = 8 = 5+3$.
The saturation with respect to $\psi_{12}$ removes
the second associated prime and yields the desired prime ideal.
\end{example}

\begin{proof}[Proof of Theorem \ref{thm:primeideal}]
We write $\CC[\psi]$ and $\CC[x]$
for the rings of polynomial functions
on $\mathcal{H}'$ and $\mathcal{V}'$ respectively.
These are polynomial  rings in $\binom{n}{d}-1$ variables, where $\psi_{[d]} = x_{[d]} = 1$.
The exponential parametrization defines an isomorphism $\iota: \CC[\psi] \rightarrow \CC[x]$
between these polynomial rings.
Note that $\iota^{-1}(x_I) = x_I(\psi)$ is the polynomial 
constructed in Proposition~\ref{prop:invertmap}.

The linear subspace  $\mathcal{V}_\sigma$ in (\ref{eq:Vsigma}) corresponds to the
ideal $P_\sigma$ in (\ref{eq:Pideal}). We write $\gamma_\sigma : \CC[x] \rightarrow \CC[x]/P_\sigma$ 
for the associated quotient map.
By definition, the prime ideal of $V_\sigma \cap \CC^{\binom{n}{d}-1}$ is the kernel of $\,\gamma_\sigma \circ \iota$. 
This equals $\iota^{-1}(P_\sigma)$, and it is a prime ideal because $P_\sigma$ is prime. Hence,
\begin{equation}
  \label{eq:iotainverse}
 \iota^{-1}(P_\sigma) 
\,\,=\,\,  \bigl\langle x_I(\psi)
\, : \, | I \backslash [d] | \in [d] \backslash \sigma  \bigr\rangle. 
\end{equation}
This is the inhomogeneous prime ideal defining the irreducible affine variety
$V_\sigma \cap \CC^{\binom{d}{n}-1}$.
We pass to the prime ideal of the projective closure $V_\sigma \subset \PP^{\binom{n}{d}-1}$
by the saturation in (\ref{eq:saturationideal}).
\end{proof}

\begin{example}[$d{=}3,n{=}6$] \label{ex:dreisechs2}
There are six distinct truncation varieties $V_\sigma$ in $\PP^{19}$.
We compute their prime ideals by the formula in (\ref{eq:saturationideal}).
  Each item is indexed by
 the corresponding set~$\sigma$:
\begin{itemize}
\item[$\{2\}$] The linear space $V_{\{2\}} \simeq \PP^9$ is the zero set of the
 ten coordinates $\psi_I$ of level $1$ or~$3$. \vspace{-0.2cm}
\item[$\{3\}$] Here we obtain the line $V_{\{3\}}  \simeq \PP^1$  that is spanned by 
the two points $e_{123}$ and $e_{456}$. \vspace{-0.2cm}
\item[$\{2,3\}$] The linear space $V_{\{2,3\}} \simeq \PP^{10}$
is the zero set of the nine coordinates $\psi_I$ of level~$1$. \vspace{-0.2cm}
\item[$\{1,2\}$] This is the cubic hypersurface $V_{\{1,2\}}$ which is given by the
master polynomial $\bar x_{456}(\psi)$. \vspace{-0.2cm}
\item[$\{1,3\}$]  The ideal $\mathcal{I}(V_{\{1,3\}})$
is generated by $25$ quadrics, and
${\rm dim}(V_{\{1,3\}}) \! =\! 10$,
${\rm deg}(V_{\{1,3\}}) \! = \! 41$.
 \vspace{-0.2cm}
\item[$\{1\}$] This is the Grassmannian $V_{\{1\}} = {\rm Gr}(3,6)$, of
dimension $9$ and degree $42$.  Its ideal is generated by $35$ quadrics.
In (\ref{eq:saturationideal}) we start with 
nine  quadrics and the cubic $\bar x_{456}(\psi)$. 
\end{itemize}
These computations were carried out with the computer algebra system
{\tt Macaulay2} \cite{M2}.
\end{example}

We have already seen Grassmannians a few times for $\sigma = \{1\}$.
This is a general result:

\begin{theorem} \label{thm:grassmann}
The truncation variety $V_{\{1\}}$ equals the
Grassmannian ${\rm Gr}(d,n)$ in its Pl\"ucker embedding in
$\,\PP^{\binom{n}{d}-1}$. The truncation varieties $V_\sigma$  are thus
generalizations of Grassmannians.
\end{theorem}

\begin{proof}
We start with an alternative characterization of the matrix $T(x)$.
Recall from \cite{FL, FO} that the    \textit{excitation operators} are the following
    graded endomorphisms of the exterior algebra:
    $$
    \chi_{\alpha,\beta}: \bigwedge \RR^n \to \bigwedge \RR^n, \quad z \mapsto (e_\alpha~\lrcorner~z) \wedge e_\beta.
    $$
    Here  $\alpha \subseteq [d]$, $\beta \subseteq [n] \backslash [d]$ and $|\alpha| = |\beta|$. The interior product $\lrcorner$ is  the operator dual to the exterior product
     \cite[Section 3.6]{GQ}.
  Informally, $e_\alpha~\lrcorner~z$ removes $e_\alpha$ from all terms of $z$, with certain sign conventions for compatibility. 
  The  excitation operators $\chi_{\alpha,\beta}$ span a linear space $\mathcal{L}$
   of dimension $\binom{n}{d}$. When restricted to 
  $\mathcal{H}$, these operators span an
  $\binom{n}{d}$-dimensional subspace of $\operatorname{Hom}(\mathcal{H}, \mathcal{H})$.
     We write elements of this subspace as $\binom{n}{d} \times \binom{n}{d}$ matrices
    \begin{equation}
    \label{eq:T(t)}
    T(t) \,\,=\,\, \sum_{\alpha, \beta} t_{\alpha,\beta} X_{\alpha,\beta}.
\end{equation}
    Here $X_{\alpha,\beta}$ is the matrix for $\chi_{\alpha,\beta}$ restricted to $\mathcal{H}$. 
    Note that $T(t)$ is our earlier matrix $T(x)$.

    We now consider the truncation to $\sigma = \{1\}$. 
    Let $\mathcal{L}_\sigma$ denote the subspace of
    $\mathcal{L}$ spanned by 
    $\,\bigl\{ \chi_{i,j}\,: \,i \in [d],\,j \in [n] \backslash [d]\bigr\}$. Operators in $\mathcal{L}_\sigma$ are derivations, i.e.~for $\tau \in \mathcal{L}_\sigma$ we have
    \begin{equation}\label{derivation}
        \tau(x \wedge y) \,\,=\,\, \tau(x) \wedge y \,+\, x \wedge \tau(y).
    \end{equation}
Let    $  T_k$ denote the matrix for the linear map $\tau \in \mathcal{L}_\sigma$ restricted to $\wedge_k \RR^n$. From (\ref{derivation}) we infer
    \begin{equation}\label{matrixderivation}
        T_{k + 1} \,\,=\,\, T_k \wedge {\rm Id}_{n} \,+\,  {\rm Id}_{\binom{n}{k}} \wedge T_1.
    \end{equation}
    We further note that $T_k$ has nilpotency $k$, that is $T_k^{k + 1} = 0$. Since the summands on the right hand side of (\ref{matrixderivation}) commute, we conclude that
    $\exp{(T_{k + 1})} = \exp(T_k) \wedge \exp(T_1)$.
    The analogous property for Kronecker sums appears in \cite[Theorem~10.9]{NH}.
   By induction on $k$, 
\begin{equation}
\label{eq:expTt}
    \exp{(T(t))} \,\,= \,\,\wedge_d \exp{(T_1(t))}.
\end{equation}

Formula (\ref{eq:expTt}) shows that
    the first column of the matrix $\exp{(T(t))}$ consists of the $d \times d$ minors of the first $d$ columns of 
the $n \times n$ matrix    
    $\exp{(T_1(t))} = {\rm Id}_n + T_1(t)$. 
    These columns~are
$$ \begin{small}
    \left[\begin{array}{cccccccc}
         1 & 0 &  \cdots & 0  & t_{1,d+1} & t_{1,d+2} &  \cdots & t_{1,n} \\
    0 & 1 &  \cdots & 0  & t_{2,d+1} & t_{2,d+2} &  \cdots & t_{2,n} \\
    0 & 0 & \ddots & 0  & \vdots & \vdots & \ddots & \vdots \\
    0 & 0 &  \cdots & 1  & t_{d,d+1} & t_{d,d+2} &  \cdots & t_{n,n} \\
    \end{array}\right]^T. \end{small}
    $$    
    This holds because the operator     $\tau \in \mathcal{L}_\sigma$ acts on the basis vectors $e_i$ of
    $\RR^n = \wedge_1 \RR^n$ as follows:
    $$     \tau e_i \,\,= \sum_{j = d + 1}^n t_{i,j} \chi_{i,j}e_i \,\,= \sum_{j = d + 1}^n t_{i,j} e_j 
    \,\,\,\, {\rm for} \,\,   \, i \leq d \quad \,{\rm and} \quad
\,    \tau e_i \, = \, 0 \,\,\,\, {\rm for} \,\,\,i > d.    $$
We conclude that
    the first column of $\exp{(T(t))}$ gives the Pl\"ucker coordinates for ${\rm Gr}(d,n)$.
    \end{proof}

\begin{remark}
All Grassmannians are obtained from
the polynomials $\bar x_I(\psi)$ with $|I \backslash [d]|  \in \{2,3,\ldots,d\}$,
by the saturation in (\ref{eq:saturationideal}). 
Starting with these polynomials
for other level sets,
we obtain all truncation varieties. This is the sense in which
the $V_\sigma$
generalize ${\rm Gr}(d,n)$.
\end{remark}

There is a natural isomorphism between the vector spaces
$\wedge_d \RR^n$ and $\wedge_{n-d} \RR^n$, and hence between
corresponding projective spaces. This 
 swaps the Grassmannians
${\rm Gr}(d,n)$ and ${\rm Gr}(n-d,n)$.
The duality extends to all truncation varieties.
This is the content of the next result,
which is our algebraic interpretation of
{\em particle-hole symmetry} in electronic structure theory.

\begin{proposition} \label{prop:duality}
Fix a subset $\sigma$ of $[d]$ and let $n \geq 2d$.
There is a linear isomorphism~between two copies of
 $\,\PP^{\binom{n}{d}-1}$ which switches the truncation varieties
$V_\sigma$ for $(d,n)$ and  for $(n-d,n)$.
\end{proposition}

\begin{proof}
The  Pl\"ucker coordinates on the two spaces are
$\psi_I$ with $|I| = d$
and $\psi_{I'}$ with $| I' | = n-d$.
Similarly, the coordinates in (\ref{eq:cpsi}) are $c_{\alpha,\beta}$ with 
$\alpha \subset [d]$ and $ \beta \subset [n] \backslash [d]$ 
for the first copy of $\PP^{\binom{n}{d}-1}$, and
 $c_{\alpha',\beta'}$ with $\alpha' \subset [n-d]$ and $\beta' \subset [n] \backslash [n-d]$
for the second copy  of $\PP^{\binom{n}{d}-1}$. The natural isomorphism
in the statement of the proposition is given by relabelling as follows:
$$ I \mapsto I' = \{ n{+}1{-}i : i \not\in I \},\,\,
\alpha \mapsto \alpha' = \{n{+}1{-}j : j \in \beta \},\,\,\,
\beta \mapsto \beta' = \{n{+}1{-}k: k \in \alpha \}. $$
One checks that our construction of  the matrix $T(x)$ and the
exponential parametrization in Section \ref{sec2} are invariant under this
  relabeling. It hence induces the desired isomorphism.
  \end{proof}
  
In light of this proposition, we shall assume $n \geq 2d$ in everything that follows.
In  Example~\ref{ex:dreisechs2} we saw a first
census of truncation varieties. We next present two further cases.

\begin{example}[$d=3,n=7$] \label{ex:dreisieben}
The six varieties correspond to the six columns in this table:
$$ \begin{matrix}
\sigma && \{1\} &\,\, \{2\} \,\,&\,\,\, \{3\} \,\,\,\,& \{1,2\} & \{1,3\} & \{2,3\} \\
{\rm dim} && 12 &  18 & 4 & 30 & 16 & 22  \\
{\rm degree} && 462 & 1 & 1 & 43 &405 & 1  \\
{\rm mingens} && \,[0,140]\, & [16] & [30] & [0,0,7] & [0,76,10] & [12] \\
{\rm CCdeg}_{3,7} &&  2883 & 19 & 5 & 1195 & 3425 & 287 \\
\end{matrix}
$$
The last row lists the CC degrees, to be introduced in Section \ref{sec5}.
The fourth row gives the number of minimal generators in degrees $1,2,3$
of the ideal $\mathcal{I}(V_\sigma)$.
All varieties live in $\PP^{34}$. The first column is the Grassmannian ${\rm Gr}(3,7)$.
Among the other five, three are linear spaces.
\end{example}

\begin{example}[$d=4,n=8$] \label{ex:vieracht}
The $14$ varieties live in $\PP^{69}$. Five of them are linear spaces:
$V_{\{3\}} \simeq \PP^{16}$, $V_{\{4\}} \simeq \PP^1$,
$V_{\{2,4\}} \simeq \PP^{37}$, $V_{\{3,4\}} \simeq \PP^{17}$,
 $V_{\{2,3,4\}} \simeq \PP^{53}$.
  The other nine are listed~here:
$$ \begin{footnotesize} \begin{matrix}
\sigma & \{1\} & \{2\} & \{1,2\} & \{1,3\} & \{1,4\} & \{2,3\} & \{1,2,3\} & \{1,2,4\} & \{1,3,4\} \\
  {\rm dim}         & 16  & 36  & 52 & 32 & 17 & 52 & 68 & 53 & 33 \\
 {\rm mingens}   & [0,721] & [32,1] & [0,\!0,63] & [0,237,200] & [0,668] & [16,1] & \! [0,\!0,\!0,\!1] \! & [0,46,120] & [0,\!236,\!200] \\
{\rm degree} &  24024 & 2 & 442066 & 24024 & 24203 & 2 & 4 & 221033 & 12012 \\
{\rm CCdeg}_{4,8} & 154441 & 73  &  ?? & 465915 & 177503 & 105 & 273  & ??& 245239
\end{matrix}
\end{footnotesize}
$$
Our methodology for
computing these degrees and CC degrees will be explained in Section~\ref{sec6}.
\end{example}

We saw in our examples that the truncation variety $V_\sigma$
is a linear space for various subsets $\sigma$ of $[d]$.
 The final theorem in this section identifies those subsets for which this happens.

\begin{theorem} \label{thm:linear}
The truncation variety $V_\sigma$ is a linear subspace
of $\,\PP^{\binom{n}{d}-1} $ if and only if  the index set
$\sigma$ is closed under addition, i.e.
if \,$i,j \in \sigma$  with $i + j \in[d]$ then $i + j \in \sigma$.
\end{theorem}

\begin{proof}
    We identify $V_\sigma$ with its restriction to  the affine chart $\mathcal{H}' = \mathbb{R}^{\binom{n}{d}-1}$.
        First assume that $\sigma$ is closed under addition. We have $1 \notin \sigma$
        because $\sigma$ is a proper subset of $[d]$. 
        Hence $\psi_I = 0$ for all $I$ of level $1$. Consider $k \geq 2$ with $k \in [d] \backslash \sigma$.
        For all  $K$ with $|K\backslash [d]| = k$ we~have        
\begin{equation}
\label{eq:funnysum}
       x_K(\psi) \,\,=\,\, \pm\,\psi_K \,+\, \sum_j a_j \,\psi_{I^{(j)}_1} \psi_{I^{(j)}_2} \cdots    \psi_{I^{(j)}_{r_j}} 
              \,\,\,=\,\,\,0.
\end{equation}
    Here $a_j \in \mathbb{Z}\backslash \{0\}$ and  $\sum_{s=1}^{r_j}
    | I^{(j)}_{s}\backslash [d]| = k$ for each $j$, 
    i.e.~the levels of the variables in each monomial sum up to $k$. 
    Since $k \in [d] \backslash \sigma$ and $\sigma$ is closed under addition, each monomial contains some $\psi_{I}$ of level $i < k$ and $i \in [d] \backslash \sigma$. By induction, $\psi_{I} = 0$. This now implies $\psi_K = 0$. 
        The restriction $V_\sigma \cap \mathcal{H}'$ is thus linear and therefore its projective closure $V_\sigma$ as well.

    Next suppose that $\sigma$ is not closed under addition. Fix the smallest $k \in  [d] \backslash \sigma$ such that $k = i + j$ for some $i,j \in \sigma$. By the same argument as above,  $\psi_I = \pm x_I(\psi) = 0$ for all $\psi_I$ of level $i < k$ where $i \in [d] \backslash \sigma$. Consider any level $k$  equation 
    (\ref{eq:funnysum}) that holds
    on the variety $V_\sigma$.
Fix any degree-compatible monomial order.  The initial monomial of $x_K(\psi)$ has degree $> 1$:
 $$   \operatorname{in}(x_K(\psi)) \,\,=\,\, 
 \psi_{I^{(j)}_1} \psi_{I^{(j)}_2} \cdots    \psi_{I^{(j)}_{r_j}}
 \qquad {\rm where} \,\,\, r_j \geq 2. $$
  Since $k$ is minimal, no monomial appearing in $x_I(\psi) \in P_\sigma = \mathcal{I}(V_\sigma)$
 divides the above initial monomial $ \operatorname{in}(x_K(\psi))$. 
Hence $\operatorname{in}(x_K(\psi))$
 is a minimal generator for the initial ideal of $\,P_\sigma$.
 This cannot be an initial ideal for a linear variety, and therefore $V_\sigma$ itself is not linear.
\end{proof}

\section{Discretization and Hamiltonians}
\label{sec4}

The electronic structure Hamiltonian is
a symmetric matrix $H$ of size $\binom{n}{d} \times \binom{n}{d}$,
describing $d$ electrons relative to a discretization with $n$ spin-orbitals.
The primary objective is to compute the quantum states that are eigenvectors of $H$, 
i.e.~chemists wish to solve the~equation 
\begin{equation}
\label{eq:eigenvector}  H \psi \,\,=\,\, \lambda \psi .
\end{equation}
In CC theory, this eigenvalue problem is replaced by a polynomial system known as the CC equations.
These will be formulated in Section \ref{sec5}.
In this section we explain where the Hamiltonian $H$ comes from.
The material that follows serves as an introductory guide for algebraists.
No background in chemistry is assumed, beyond what is taught in high school.

The starting point of our analysis is the {\em electronic Schrödinger equation}
\begin{equation} \label{eq:schroedinger}
\mathscr{H} \,\Psi(\mathbf{x}_1,\vx_2,\ldots ,\mathbf{x}_d) \,\,=\,\, 
\lambda \,\Psi(\mathbf{x}_1,\vx_2 \ldots,\mathbf{x}_d).
\end{equation}
From this differential equation, we shall derive a finite-dimensional eigenvalue problem (\ref{eq:eigenvector}).
The unknown in (\ref{eq:schroedinger}) is
the {\em wave function} $\Psi(\mathbf{x}_1,\vx_2,\ldots,\mathbf{x}_d)$.
The arguments in this function are pairs $\vx_i = (\vr_i, s_i)$ where
$\vr_i= (r_i^{(1)},r_i^{(2)},r_i^{(3)})$ are points in $\RR^3$ that represent the positions of
$d$ electrons, and $s_i\in \{\pm1/2\}$ describes the electronic spin (vide infra). We assume that $\Psi$ is sufficiently differentiable~\cite{Y}.
By {\em Pauli's exclusion principle}, the wave
function $\Psi$ must be antisymmetric in its $d$ arguments, i.e.~if ${\bf x}_i$ and
${\bf x}_j$ are switched then $\Psi$ is replaced by $-\Psi$.
The Hamiltonian   $\mathscr{H}$ in (\ref{eq:schroedinger}) is a second order differential operator
which describes the behavior of $d$ interacting electrons in the vicinity of $d_{\rm nuc}$ stationary nuclei.
The formula~is
\begin{equation}
\label{eq:ElecSE}
\mathscr{H}
\,\,=\,\,
-\,\frac{1}{2}\sum_{i=1}^d\Delta_{\vr_i}
\,-\,\sum_{i=1}^d\sum_{j=1}^{d_{\rm nuc}}\frac{Z_j}{\vert \vr_i-\vR_j\vert} 
\,+\,\sum_{i=1}^d\sum_{j=i+1}^d\frac{1}{\vert \vr_i-\vr_j\vert}.
\end{equation}

The symbol $\Delta_{\vr_i}$ in the leftmost sum denotes the Laplacian
$\, \sum_{j=1}^3 (\partial /\partial r_i^{(j)} )^2$.
All other summands in (\ref{eq:ElecSE}) act  on $\Psi$ by multiplication.
They contain constants which we now explain.
The atoms and their nuclei
are indexed by $j=1,2,\ldots,d_{\rm nuc}$.
The constant $Z_j$ is the $j$th nuclear charge.
This is the atomic number listed in the periodic table, i.e. $Z_j$ is a positive integer.
The position of the $j$th nucleus is the point $\vR_j \in \mathbb{R}^3$ which is also constant.
We here consider only charge-neutral molecules.
This means that $d=\sum_{j=1}^{d_{\rm nuc}} Z_j$ and $d_{\rm nuc} \leq d$. 

The following molecule will serve as our running example, both here and in Section \ref{sec6}.

\begin{example}[Lithium hydride]
This molecule has the formula LiH. It involves
$d_{\rm nuc} \! =\! 2$ atoms, namely
lithium Li and hydrogen H. Their atomic numbers are $Z_1 =3$ and
$Z_2 = 1$, so the number of electrons is $d = Z_1 + Z_2 = 4$.
The two nuclei are fixed at locations $\vR_1$ and $\vR_2$, whereas the
four electrons have variables $\vx_1,\vx_2,\vx_3,\vx_4$.
Here, $\Psi$ is an unknown function of $16$ scalars. It satisfies
$\Psi(\vx_1,\vx_2,\vx_3,\vx_4) = - \Psi(\vx_2,\vx_1,\vx_3,\vx_4) =  \cdots =
-\Psi(\vx_1,\vx_2,\vx_4,\vx_3)$.
\end{example}

The next step is to construct a finite-dimensional space of functions, along
with a suitable basis, which contains
 an approximate solution to the electronic Schr\"odinger equation 
(\ref{eq:ElecSE}). The restriction of $\mathscr{H}$ to that finite-dimensional space
will be our symmetric matrix $H$.

 There are many ways to select a suitable basis, and hence a discretization of $\mathscr{H}$.
 See~\cite{LBV} for an overview. We apply the method called
{\em linear combination of atomic orbitals} (LCAO). This is used widely 
in quantum chemistry. The LCAO method starts with {\it atomic orbitals}.

We select a set 
 $ \bigl\{ \chi_1,\chi_2,\ldots,\chi_k\}$
 of atomic orbitals.
 These are sufficiently smooth functions $\chi_i : \RR^3 \rightarrow \RR$
 which are linearly independent. 
 Atomic orbitals encompass substantial physical principles.
 Readers can refer to~\cite[Chapters 5, 6 and 8]{HJO} for a comprehensive explanation and motivation. 
Notably, atomic orbital basis sets for different atoms are well-documented and available through online data resources like
 \url{www.basissetexchange.org}. The number $k$ of
 atomic orbitals is greater than or equal to the number $d$ of
 electrons, i.e.~$d \leq k$. The equality $d=k$ can hold.
The number $k$ of atomic orbitals determines the total count $n$ of one-particle basis functions used in the discretization. We shall be led shortly to $n=2k$.

\begin{example}[$d=k=4$] \label{ex:revisitLiH}
Revisiting the lithium hydride (LiH) molecule, we
select $k=4$ atomic orbitals, three for lithium and one for hydrogen.
A graphical representation of these four atomic orbitals is shown in
Figure~\ref{fig:AOsPlot}.
The pictures are iso-surfaces
for $\chi_1,\ldots,\chi_4$. An {\em iso-surface} has the form
$\{ {\bf r} \in \RR^3: \chi_i({\bf r}) = c \}$,
for a constant $c$ that can be positive or~negative.

\begin{figure}[!ht]
    \centering
    \includegraphics[width = \textwidth]{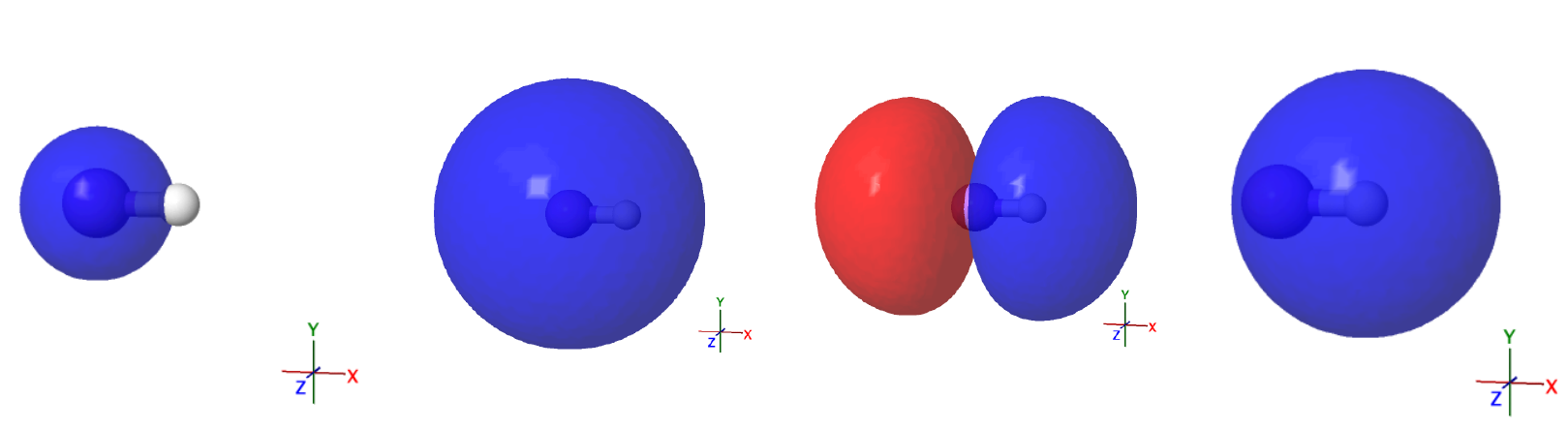}
    \caption{
    Iso-surfaces of four atomic orbitals (Li: 1s, 2s, 2p$_0$; H: 1s) of lithium hydrate. The iso-surfaces correspond to the value $\pm 0.025$, with blue for $c = +0.025$ and red for $c = -0.025$.}
    \label{fig:AOsPlot}
\end{figure}
\end{example}

Before proceeding with the LCAO approach, we need to account for the 
{\em electronic spin}, a crucial degree of freedom in electronic structure theory that distinguishes  physical states.
The inclusion of spin doubles the size of our basis: 
each atomic orbital $\chi_i$ will be replaced by two functions.
Even though the electronic Schr\"odinger equation~\eqref{eq:ElecSE} does not explicitly include the electronic spin, it remains a significant factor in electronic structure calculations~\cite{Y}. 
This explains the equation $n=2k$ we stated before Example \ref{ex:revisitLiH}.
Note that $n \geq 2d$.

The electronic spin can take two possible values: ``spin up'' (+1/2) or ``spin down''~(-1/2). 
To incorporate this, we introduce a spin variable $s\in\{\pm 1/2\}$ and two binary functions
$$
m_0(s) \,=\, 
\left\lbrace
\begin{split}
1 \quad &{\rm if}~ s=+1/2,\\
0 \quad &{\rm if}~ s=-1/2,
\end{split}
\right.
\quad 
{\rm and}
\quad
m_1(s)\, =\, 
\left\lbrace
\begin{split}
1 \quad &{\rm if}~ s=-1/2,\\
0 \quad &{\rm if}~ s=+1/2.
\end{split}
\right.
$$

The atomic orbitals can be separated into spatial and spin components, namely
\begin{equation}
\phi_i(\vr, s) \,\, = \,\,
\chi_i(\vr)\,   m_0(s) 
\quad
{\rm and}
\quad
\phi_{k+i}(\vr, s) \,\, = \,\,
\chi_i(\vr) \,  m_1(s) \quad {\rm for} \,\,\, i=1,2,\ldots,k.
\end{equation}
This factorization simplifies the treatment of electronic spin, making it possible to handle the spatial and spin degrees of freedom independently in calculations.
In order to simplify notation, we replace $\mathbb{R}^3$ by
$\,X := \mathbb{R}^3 \times \{\pm 1/2 \}$, and
we introduce compound coordinates  $\vx = (\vr, s)$ on $ X$.
Using these, we equip the atomic orbital space on $X$ with the inner product 
\begin{equation}
\label{eq:InnerProductMOs}
\langle
\phi_i, \phi_j
\rangle_{L^2(X)}
\,\,:=\,\,
\int \! \phi_i (\vx) \phi_j(\vx) d\vx \,\,\, = \!
\sum_{s \in \{\pm 1/2\}} \!\!\!\!\! m_{\lfloor \frac{i}{k+1}\rfloor}(s)\, m_{\lfloor \frac{j}{k+1}\rfloor}(s) \!
\int_{\mathbb{R}^3}\! \chi_{i} (\vr)\, \chi_{j}(\vr) d\vr,
\end{equation}
where the indices of $\chi_i$ and $\chi_j$ in the right integral are understood as $i$ and $j$
modulo $k+1$.

In principle, we find a {\em Galerkin basis} for $\mathscr{H}$ by 
passing from  the $\phi_i$ to $d$-particle functions. 
However, the LCAO method introduces an additional set of orthonormal functions known as {\it molecular spin orbitals}, which describe the behavior of individual electrons within the molecule. 
The molecular orbitals resemble the atomic orbitals. They are linear combinations:
\begin{equation} \label{eq:hartreefock}
\xi_i \,\,= \,\,\sum_{j=1}^{n} \,C_{j,i} \phi_j \qquad \hbox{for} \quad i=1,2,\ldots,n.
\end{equation}
The functions $\xi_j$ and $\phi_i$ span the same 
$n$-dimensional vector space. 
The determination of the expansion coefficients $C_{j,i}$ typically involves employing 
{\em Hartree-Fock theory}, as explained in~\cite[Chapter 10]{HJO} or \cite[Chapter 2.1]{LJ}.
For our specific example, lithium hydride, the molecular orbitals obtained from (spin-restricted) Hartree-Fock theory are depicted in Figure~\ref{fig:MOsPlot}.
 
\begin{figure}[!ht]
    \centering
    \includegraphics[width = \textwidth]{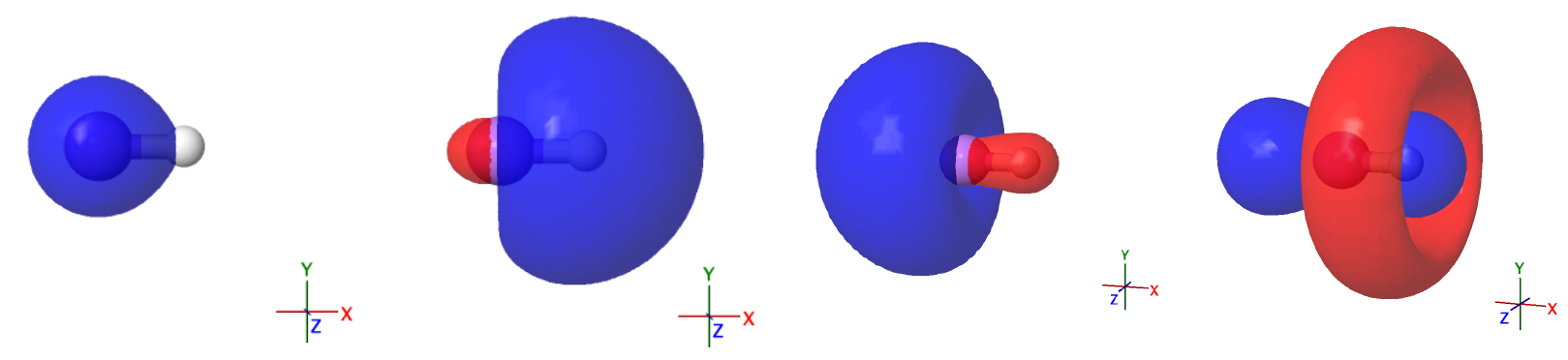}
    \caption{Iso-surfaces of four molecular orbitals of lithium hydrate for $s=+1/2$. The iso-surfaces correspond to the value $\pm 0.04$, with blue for $+0.04$ and red for $-0.04$.}
    \label{fig:MOsPlot}
\end{figure}

We shall now employ the molecular orbitals to construct $d$-particle functions.
Since we are investigating electrons (i.e. fermions, which must adhere to Pauli's exclusion principle),
 the $d$-particle functions must be anti-symmetric~\cite[Chapter 4.2]{Y}.
Starting from the molecular orbitals, we form a Galerkin basis of $\binom{n}{d}$ anti-symmetric $d$-particle functions
as follows:
$$
\Phi_I(\vx_1,\ldots,\vx_d)=
\Phi_{i_1, \ldots,i_d}(\vx_1,\ldots,\vx_d) \,= \,
\frac{1}{\sqrt{d!}}
(\xi_{i_1} \wedge  \cdots \wedge \xi_{i_d})(\vx_1,\ldots,\vx_d) \,\,\, \hbox{where} \,\,
I = \{i_1,\ldots,i_d\}.
$$

In order to calculate matrix elements, we equip the Galerkin space with the inner product 
\begin{equation}
\label{eq:innerprod}
\langle \Phi_I, \Phi_J \rangle_{\left(L^2(X)\right)^d}
\,\,=\,\, \prod_{p=1}^d \,\langle \,\xi_{i_p}\, ,\, \xi_{j_p}\, \rangle_{L^2(X)}.
\end{equation}
The evaluation of this inner product requires the numerical evaluation of non-trivial integrals.
A common way to circumvent this intricate numerical integration is to use Gaussian-type atomic orbitals~\cite[Chapter 8]{HJO}.
With this, we can proceed to discretize the Hamiltonian $\mathscr{H} $.
The following real numbers are the entries in our symmetric matrix $H$
of format $\binom{n}{d} \times \binom{n}{d}$:
\begin{equation}
\label{eq:matrixElem}
H_{I,J} \,\,= \,\,
\langle 
\Phi_I, \mathscr{H} \Phi_J
\rangle_{\left(L^2(X)\right)^d}.
\end{equation}
This matrix $H$ is known as the {\em electronic structure Hamiltonian in first quantization}. 
Its entries are parameters in the CC equations which will be introduced in the next section.

\begin{example}[Lithium hydride] \label{ex:morerunning}
In our running example, we have  $d=4$ and $n=2d=8$.
Our Hamiltonian $H$ for LiH is a symmetric $70 \times 70$ matrix
which is computed as follows.
Using PySCF~\cite{PyScf1}, 
we evaluate the following integrals 
for all indices $p,q,r,s \in [8]$:

\begin{footnotesize} \vspace{-0.2in}
$$
h_{p,q}
=
\int
\xi_{p}(\vx) \left(
-\frac{\Delta_{\vr}}{2} - \sum_{j}\frac{Z_j}{| \vr - \vR_j|}
\right)\xi_{q}(\vx)d\vx
\quad
{\rm  and} 
\quad
v_{p,q,r,s}
=
\int
\frac{
\xi_p(\vx) \xi_q(\vx) \xi_r(\vx') \xi_s(\vx')
}{|\vr - \vr'|} d\vx d\vx'.
$$
\end{footnotesize}
\noindent
This uses the molecular orbital basis $\{\xi_1(\vx),\ldots, \xi_8(\vx)\}$ 
obtained from Hartree-Fock theory.
For computing the matrix elements in (\ref{eq:matrixElem}), we 
expand the $d$-particle functions $\Phi_I$ and~$\Phi_J$ using antisymmetry:
$$
\Phi_I(\vx_1,\ldots,\vx_4)
\,\,=\,\,
\frac{1}{\sqrt{24}}
\sum_{\pi \in S_I}
{\rm sgn}(\pi) \xi_{\pi(i_1)}(\vx_1)\cdot \ldots \cdot  \xi_{\pi(i_4)}(\vx_4).
$$
The fact that each molecular orbital is a function in one variable allows us to factorize the integral expression. 
In this factorized expression $h_{p,q}$ and $v_{p,q,r,s}$ will make their appearance. 
Finally, we use the orthogonality of
the molecular orbitals with respect to the inner product defined in~(\ref{eq:InnerProductMOs}).
This yields the following expression for the entries of the desired matrix:

\begin{footnotesize} \vspace{-0.2in}
\begin{equation*}
\begin{aligned}
H_{I,J} \,=  &
\sum_{
\substack{
\rho \in S_I\\
\pi\in S_J}} \!
{\rm sgn}(\rho) \,
{\rm sgn}(\pi)
\sum_{\ell=1}^4
\biggl(
h_{\rho(i_\ell),\pi(j_\ell)} \!
\prod_{\ell \neq k = 1}^4 \delta_{\rho(i_k), \pi(j_k)}
\,\,+\,\,
\sum_{j>\ell}^4
v_{\rho(i_\ell),\pi(j_\ell),\rho(i_j),\pi(j_j)}
\! \prod_{\ell \neq j \neq k = 1}^4 \!\! \delta_{\rho(i_k), \pi(j_k)}
\biggr).
\end{aligned}
\end{equation*}
\end{footnotesize}

\noindent
In conclusion, by means of PySCF, we compute a $70 \times 70$ matrix $H$,
to be used in Example~\ref{ex:ending48}.
\end{example}

\section{The CC Equations}
\label{sec5}

We now present the coupled cluster (CC) equations. 
These approximate the eigenvalue problem (\ref{eq:eigenvector}). 
The ambient space will be one of the truncation varieties~$V_\sigma$. 
The equations are determined by a Hamiltonian $H$ as in Section \ref{sec4}.
A first systematic study, with a focus on Newton polytopes, was undertaken in \cite{FO}. 
In what follows we derive the CC equations from the perspective of algebraic geometry,
 leading to an alternative description. 
In Theorem \ref{thm:twoformulations} we examine to what extent our formulation here is equivalent to the one seen in \cite{FL, FO}.

Let $\sigma $ be a non-empty proper subset of $[d]$.
The set of indices with level in $\sigma$ is denoted
$$ \begin{matrix} \widetilde \sigma  \, = \, \bigl\{  I \in \binom{[n]}{d} \, : \,| I \backslash [d] | 
 \in \sigma \bigr\}. \end{matrix} $$
 We already saw that the dimension of the truncation variety $V_\sigma$ equals the cardinality $| \widetilde \sigma |$.
For the Grassmannian $V_{\{1\}} = {\rm Gr}(d,n)$, the set $\widetilde{ \{1 \}}$ consists of all $d$-tuples that differ from $[d] $ in exactly one index, so we have $| \widetilde{\{1\}} | = d(n-d) = {\rm dim}({\rm Gr}(d,n))$.
For $\psi \in \mathcal{H}$, we denote by $\psi_\sigma $ the truncation of
the vector $\psi$ to the coordinates $\psi_I$ where 
$I \in   \widetilde \sigma \cup \{[d] \}$.
Our problem~is:
\begin{equation}
\label{eq:CC1}
\hbox{
Compute all
$\psi \in V_\sigma$ with
$\psi_{[d]} \not= 0$ such that
$(H \psi)_\sigma$ and $\psi_\sigma$ are linearly dependent.}
\end{equation}
This translates into a system of quadratic equations on the projective variety $V_\sigma$. 
This is the system of {\em CC equations}. 
The linear dependency condition in (\ref{eq:CC1}) can be expressed by the $2 \times 2$ minors of the $ (|\widetilde \sigma| {+} 1) \times 2$ matrix $\bigl[ (H \psi)_\sigma \, ,\psi_\sigma \bigr]$.
This represents the truncation of~(\ref{eq:eigenvector}):
\begin{equation}
\label{eq:trunceigen}
(H \psi)_\sigma \,\,=\,\, \lambda \,\psi_\sigma . 
\end{equation}
The number  $|\widetilde \sigma|$ of constraints imposed by this equation equals the dimension of $V_{\sigma}$, so we expect the CC equations to have a finite number of solutions for generic $H$. This is indeed the case. 
We call this number the {\em CC degree} of the variety $V_\sigma$, denoted by ${\rm CCdeg}_{d,n}(\sigma)$.

\begin{example}[$d=2, \sigma =\{1\} $] \label{ex:G2nfirst}
Consider the CC equations on the Grassmannian ${\rm Gr}(2,n)$.
This has $\binom{n}{2}$ Pl\"ucker coordinates $\psi_{ij}$, and $\psi$ is the column vector 
with these coordinates. The truncated column vector $\psi_\sigma$ has $2n-3$ entries, 
namely all coordinates $\psi_{ij}$ where $\{i,j\} \cap \{1,2\} \not= \emptyset$.
These coordinates form a transcendence basis for the coordinate ring of ${\rm Gr}(2,n)$.
The Hamiltonian $H$ is a matrix of format $\binom{n}{2} \times \binom{n}{2}$.
Then $H \psi$  and $(H \psi)_\sigma$ are column vectors of length $\binom{n}{2}$ and $2n-3$ respectively.
The CC equations (\ref{eq:CC1}) impose $2n-4$  conditions
on the $(2n-4)$-dimensional variety ${\rm Gr}(2,n)$.
See Conjecture \ref{conj:G2n}
for the number of solutions.
\end{example}

The CC degree is the number of solutions to the CC equations.
Some non-trivial values of ${\rm CCdeg}_{d,n}(\sigma)$
 were  shown in the tables of
Examples \ref{ex:dreisieben} and \ref{ex:vieracht}.
We have the following general upper bound for the CC degree
in terms of the degree of the truncation variety $V_\sigma$.

\begin{theorem} \label{thm:generalbound}
For any Hamiltonian $H$, the number of isolated solutions to 
(\ref{eq:CC1}) satisfies
\begin{equation}
\label{eq:CCbound}
 {\rm CCdeg}_{d,n}(\sigma) \quad \leq \quad
\bigl({\rm dim}(V_\sigma) + 1 \bigr)  \, {\rm deg}(V_\sigma). 
\end{equation}
\end{theorem}

\begin{proof}
We set $N = |\widetilde \sigma| = {\rm dim}(V_\sigma)$.
For generic Hamiltonians $H$,
the variety in $\PP^{\binom{n}{d}-1}$ defined by the $2 \times 2 $ minors of
the $(N{+}1) \times 2$ matrix $\bigl[ (H \psi)_\sigma \, ,\psi_\sigma \bigr]$
has codimension $N$
and degree
$N{+}1 $. Geometrically, this variety is a cone over a generic section of
the Segre variety $\PP^1 \times \PP^N  \subset \PP^{2N+1}$.
The first factor on the right-hand side in~(\ref{eq:CCbound})
is $N+1 = {\rm deg}(\PP^1 \times \PP^N)$.
The intersection with $V_\sigma$ has expected dimension zero,
but it can  have higher-dimensional components.
We are interested in the number of isolated solutions.
 By B\'ezout's Theorem, this number is at most the product of the degrees of the two varieties,
 seen on the right in (\ref{eq:CCbound}).
That bound on the number of isolated solutions also holds for special matrices $H$.
\end{proof}

We note that the equality holds in (\ref{eq:CCbound}) for the linear
cases described in Theorem \ref{thm:linear}.

\begin{corollary} \label{cor:usual}
Suppose that  $\,V_\sigma$ is a linear space. Then
$ {\rm CCdeg}_{d,n}(\sigma) = {\rm dim}(V_\sigma) + 1$,
we have $(H \psi)_\sigma = H_{\sigma,\sigma}   \psi_\sigma$,
and (\ref{eq:trunceigen}) is the eigenvalue problem for the symmetric
 matrix $H_{\sigma,\sigma}$.
In particular, all complex solutions to the CC equations (\ref{eq:CC1}) are~real.
\end{corollary}

\begin{proof}
The vector $\psi$ is zero in all coordinates outside $\widetilde \sigma \cup \{[d]\}$,
and it is arbitrary otherwise, since $V_\sigma = \PP^{|\widetilde{\sigma}|}$.
This implies $(H \psi)_\sigma = H_{\sigma,\sigma}   \psi_\sigma$,
and the other assertions follow from this.
\end{proof}

We expect  (\ref{eq:CCbound}) 
to be strict whenever ${\rm deg}(V_\sigma) \geq 2$. This holds
whenever $ {\rm CCdeg}_{d,n}(\sigma) $ is known.
There is a geometric explanation:
the intersection of $V_\sigma$ with the variety of
$2 \times 2$ minors of  $\bigl[ (H \psi)_\sigma \, ,\psi_\sigma \bigr]$
 is not transverse on the hyperplane $V(\psi_{[d]})$.
 Finding the true CC degree is a problem in intersection theory,
 just like computing the degree of $V_\sigma$ itself.
  The two numbers are important for quantum chemistry
 because they govern the complexity of
 solving the CC equations. In particular, $  {\rm CCdeg}_{d,n}(\sigma) $ is
 the number of paths to be tracked when solving numerically with
  {\tt HomotopyContinuation.jl}. This is discussed in
  Section~\ref{sec6}.

 Our next example shows 
the degrees and CC degrees  for the simplest Grassmannian case.

\begin{example}[$d=2, \sigma = \{1\}$]  \label{ex:G2nsecond}
We continue Example \ref{ex:G2nfirst}.
The degree of the Grassmannian
${\rm Gr}(2,n)$ is the Catalan number
$\frac{1}{n-1} \binom{2n-4}{n-2}$.
We see this in \cite[equation (5.5) and Theorem 5.13]{MS}.
The degree of the variety of $2 \times 2 $ minors is $ 2n-3$.
Hence the upper bound in (\ref{eq:CCbound}) equals
$\frac{2n-3}{n-1} \binom{2n-4}{n-2} = \binom{2n-3}{n-1}$.
This binomial coefficient is $10,35,126,462,1716,6435$
for $n=4,5,6,7,8,9$. Using computational methods, we found
that the true CC degrees are $9,27,83,263,857,2859$
for $n=4,5,6,7,8,9$. This motivates the following formula.\end{example}

\begin{theorem}
\label{conj:G2n}
The CC degree associated with the Grassmannian ${\rm Gr}(2,n)$ equals
$$
 {\rm CCdeg}_{2,n}(\{1\}) \,\, = \,\, \frac{4}{n} \binom{2n-3}{n-1} \, - \,1. 
$$
\end{theorem}

This result was stated as a conjecture in the first version of
this article from August 2023. It was proved in October 2024
in collaboration with Viktoriia Borovik. It is published in \cite{BSS}.

The inequality   (\ref{eq:CCbound}) is strict for ${\rm Gr}(2,n)$
because the CC equations admit
extraneous solutions that lie
on the  hyperplane $\{\psi_{12} = 0\}$.
These form a higher-dimensional component which can
be removed by saturation as in (\ref{eq:saturationideal}).
We discuss this in detail for a small instance.

\begin{example}[$d=2,n=5$]
The CC equations are written via rank constraints as follows:
$$
{\rm rank}
\begin{bmatrix}
0 & \psi_{12} & \psi_{13} & \psi_{14} & \psi_{15} \\
-\psi_{12} & 0 & \psi_{23} & \psi_{24} & \psi_{25} \\
-\psi_{13} & - \psi_{23} & 0 & \psi_{34} & \psi_{35} \\
-\psi_{14} & - \psi_{24} & -\psi_{34} & 0 & \psi_{45} \\
-\psi_{15} & - \psi_{25} & -\psi_{35} & - \psi_{45} & 0
\end{bmatrix}
\,\leq \, 2
\qquad {\rm and} \qquad
{\rm rank}
\begin{small}
\begin{bmatrix}
\,(H \psi)_{12} & \psi_{12} \,\, \\
\,(H \psi)_{13} & \psi_{13} \,\, \\
\,(H \psi)_{14} & \psi_{14} \,\, \\
\,(H \psi)_{15} & \psi_{15} \,\, \\
\,(H \psi)_{23} & \psi_{23} \,\, \\
\,(H \psi)_{24} & \psi_{24} \,\, \\
\,(H \psi)_{25} & \psi_{25} \,\,
\end{bmatrix}
\end{small}
\, \leq \, 1.
$$
Indeed, the Grassmannian $V_{\{1\}} = {\rm Gr}(2,5)$
is cut out in $\PP^9$ by the $4 \times 4$ Pfaffians in
a skew-symmetric $5 \times 5$ matrix (see \cite[Example 4.9]{MS}).
Let $\mathcal{I}$ be the ideal generated by these $5$
Pfaffians plus the $\binom{7}{2} = 21$ maximal minors of the $7 \times 2$ matrix on the right.
This gives the upper bound $35 = 5 \times 7$ in Theorem \ref{thm:generalbound}.
The ideal $\mathcal{I}$ is radical, and it is the intersection of
the desired ideal of codimension $9$ and a linear ideal
of codimension~$7$, namely
$\,\langle \psi_{12},
 \psi_{13}, \psi_{14}, \psi_{15}, \psi_{23}, \psi_{24}, \psi_{25} \rangle$.
This extraneous component is responsible for the differerence $8$
between the upper bound and the true CC degree, which is 
$27$, as in Theorem~\ref{conj:G2n}.
 \end{example}

\begin{example}[$d=3,n=6$]  \label{ex:dreisechs3}
We consider the six truncation varieties $V_\sigma $
in Example~\ref{ex:dreisechs2}:
$$
\begin{matrix} 
\sigma && \{1\} & \,\,\{2\}\,\, &\,\, \{3\} \,\,& \{1,2\} & \{1,3\} & \{2,3 \} \\
{\rm dim} && 9 & 9 & 1 & 18 & 10 & 10 \\
{\rm degree} && 42 & 1 & 1 & 3 & 41 & 1 \\
{\rm bound} && 420 & 10 & 2 & 57 & 451 & 11 \\
{\rm CCdeg}_{3,6} && 250 & 10 & 2 & 55 & 420 & 11 \\
\end{matrix}
$$
In three cases, the variety $V_\sigma$ is a linear space
in $\PP^{19}$, and the CC degree is ${\rm dim}(V_\sigma)+1$.
In the other three cases, the CC degree is a bit below the bound
given by Theorem \ref{thm:generalbound}.
See Examples~\ref{ex:dreisieben} and~\ref{ex:vieracht} for
more comparisons between our upper bound and the CC degree.
\end{example}

Of special interest is the case when the truncation variety $V_\sigma$ is a hypersurface.
This is the hypersurface 
defined by the master polynomial 
 $t_{[d],[\overline d]}(c) = x_{[2d] \backslash [d]} (\psi) $.
 This polynomial was shown explicitly in Theorem \ref{thm:UBP}.
Our next result gives the CC degree for this hypersurface.

\begin{proposition} \label{prop:offby}
If $n=2d$ and $\sigma = \{1,2,\ldots,d-1\}$, then the bound (\ref{eq:CCbound})
 is off by $d-1$:
\begin{equation}
\label{eq:CCmaster}
{\rm CCdeg}_{d,2d}(V_\sigma)  \,\, = \,\,
\bigl({\rm dim}(V_\sigma) + 1 \bigr)\,   {\rm deg}(V_\sigma) \,-\, (d-1)
\,\,\, = \,\,\, d \binom{2d}{d} - 2d + 1.
\end{equation}
\end{proposition}

\begin{proof}
We wish to count all scalars $\lambda \in \CC$ such that
the truncated eigenvalue equation
$\, \bigl[ (H-\lambda \,{\rm Id})   \psi \bigr]_\sigma = 0 $ has
a solution $\psi$ in $V_\sigma$.
For this, we delete the last row of the matrix 
$H-\lambda \,{\rm Id}$ to get a matrix with one more column than rows.
Using Cramer's Rule, we write the entries of $\psi$
as signed maximal minors of that matrix. The last entry of
$\psi$ is a polynomial in $\lambda$ of degree $\binom{2d}{d}-1$,
which is the size of these minors. All other entries of $\psi = \psi(\lambda)$
are polynomials of degree $\binom{2d}{d}-2$,
because $\lambda$ does not occur in the last column of our matrix.

We substitute  the vector $\psi = \psi(\lambda)$ into
the equation $f = x_{[2d]\backslash [d]} (\psi)$
that defines $V_\sigma$.
We know that $f$ has degree $d$ and it is linear in the
last variable $\psi_{[2d]\backslash [d]}$.  This implies that
$f(\psi(\lambda))$ is a polynomial in one variable $\lambda$ of degree
$\bigl(\binom{2d}{d}-2 \bigr)(d-1) + \bigl(\binom{2d}{d}-1\bigr) $.
Since $H$ is generic, the polynomial is 
square free,
and its number of complex zeros is given in
(\ref{eq:CCmaster}).
\end{proof}

We next express the CC equations in terms of the cluster amplitudes $x_I$.
Let $z$ denote the restriction of the vector $x$ to the coordinates in $\widetilde \sigma$.
To be precise, $z$ is the vector of length $\binom{n}{d}$ which is obtained from $x$ by setting $x_{[d]} = 1$ and $x_J = 0$ for all $J \not\in \widetilde \sigma$.
We identify $\CC^{|\widetilde \sigma|}$ with the space of such vectors $z$.
The truncation variety $V_\sigma $ has the parametric representation
$$ \CC^{| \widetilde \sigma |}\, \rightarrow\, \PP^{\binom{n}{d}-1},\,\,
z \,\mapsto \,{\rm exp}(T(z))   e_{[d]}. $$
We substitute this parametrization into 
 the $ (|\widetilde \sigma| {+} 1) \times 2$ matrix $\bigl[ (H \psi)_\sigma \, ,\psi_\sigma \bigr]$
 seen in  (\ref{eq:CC1}).
  Therefore the CC equations  (\ref{eq:trunceigen})
 are equivalent to the following equations in the unknowns~$z$:
\begin{equation}
\label{eq:formulation1}
{\rm rank}\, \bigl[ \,H  \, {\rm exp}(T(z))  \, e_{[d]} \,\,| \,\,
 {\rm exp}(T(z))   e_{[d]} \,\bigr]_{\sigma} \,\, \leq \,\, 1.
\end{equation}
Here we require  column vectors of length $|\widetilde \sigma|+1$
to be linearly dependent. For computations, it is advantageous
to rewrite (\ref{eq:formulation1}) as a square system of 
$|\widetilde \sigma|+1$ equations in $|\widetilde \sigma|+1$ unknowns:
\begin{equation}
\label{eq:formulation1b}
 \bigl[ \, (H - \lambda\, {\rm Id})   \,{\rm exp}(T(z))   e_{[d]}  \,\bigr]_{\sigma} \,=\,0
\qquad \hbox{for some} \,\, \lambda \in \CC. 
\end{equation}

We now compare the system (\ref{eq:formulation1b}) with the traditional formulation of the CC equations,
which was used in \cite{FO} and in earlier works. That formulation is
based on the observation
\begin{equation}
\label{eq:inversematrix} {\rm exp}(T(z))^{-1} \,\, = \,\, {\rm exp}(T(-z)). 
\end{equation}
Before truncation, one left-multiplies \eqref{eq:formulation1} by the 
matrix inverse (\ref{eq:inversematrix}) to get
\begin{equation}
\label{eq:formulation2}
{\rm rank}\, \bigl[\, {\rm exp}(T(-z))   \,H  \, {\rm exp}(T(z))  \, e_{[d]} \,\,| \,\,
e_{[d]} \,\bigr]_{\sigma} \,\, \leq \,\, 1.  
\end{equation}
Since $e_{[d]}$ is a unit vector, this is actually a square system
of $|\widetilde \sigma|$ equations in $|\widetilde \sigma|$~unknowns:
\begin{equation}
\label{eq:formulation2b}
\bigl[\, {\rm exp}(T(-z))  \, H \,  {\rm exp}(T(z))  \, e_{[d]} \,\bigr]_{\sigma}\,=\, 0.
\end{equation}  

\begin{remark}
The square system in (\ref{eq:formulation2b}) is
equivalent to the CC equations that are presented in
\cite[equation (4.2)]{FO}.  
The Newton polytopes of these equations are studied
in \cite[Section 4.1]{FO},  and
\cite[Section 4.2]{FO} offers
a reformulation as quadratic equations
in more variables.
\end{remark}

It turns out that -- sometimes -- the traditional formulation
(\ref{eq:formulation2b}) yields a polynomial system
that is fundamentally
 different from the system we derived in (\ref{eq:formulation1b}).
The reason for this discrepancy is that
the left multiplication with (\ref{eq:inversematrix}) need 
not commute with truncation.

\begin{example}
Let $d=3, n=6$ and $\sigma = \{2,3\}$.
The variety $V_\sigma$ is a linear space
$\PP^{10}$, and (\ref{eq:formulation1})-(\ref{eq:formulation1b})
is an ordinary eigenvalue problem for the $11 \times 11 $ matrix $H_{\sigma,\sigma}$.
It has $11$ solutions, all real.
On the other hand, the system
(\ref{eq:formulation2})-(\ref{eq:formulation2b}) has $20$ complex solutions.
Experimentally, the number of real solutions ranges between $6$ and $14$.
The two systems are genuinely different.
\end{example}

The following result says that the discrepancy  we discovered is
actually not so bad. It characterizes all CC variants where
the old and new formulation of the CC equations agree.

    \begin{theorem} \label{thm:twoformulations}
    The system (\ref{eq:formulation1})-(\ref{eq:formulation1b}) is equivalent to the
     system (\ref{eq:formulation2})-(\ref{eq:formulation2b})
    if and only if the set $\sigma$ has the form
    $\,m [k] \, = \{m,2m,\ldots,km\}\,$
    for some positive integers $m,k$ with $km \leq d$.
    \end{theorem}

    \begin{proof}
    We consider a matrix $A$ whose rows and columns are indexed by $\binom{[n]}{d}$.
    The identity
        $$
        [\, AB\,]_{\sigma} \,\,=\,\, [\,A\,]_{\sigma,\sigma}   [\,B\,]_{\sigma}
        $$
        holds for a general matrix $B$ if and only if $[\,A\,]_{\sigma,\sigma^c} = 0$, where $\sigma^c = [d] \backslash \sigma$. Suppose that 
\begin{equation}
\label{eq:blockzero}        
        [\,\exp{(T(z))}\,]_{\sigma,\sigma^c} \,\,=\,\, 0.         \end{equation}
Then the same block is zero in the inverse matrix $\exp{(-T(z))}$,
         and therefore
                 \begin{align*}
            [\, (H - \lambda\, {\rm Id})  \exp(T(z))   e_{[d]} \,]_{\sigma}
            &\,\,=\,\, [\, \exp{(-T(z))} \,]_{\sigma,\sigma}   [\, (H - \lambda\, {\rm Id})  \exp(T(z))   e_{[d]} \,]_{\sigma}\\
            &\,\,=\,\, [\, \exp{(-T(z))}   H   \exp(T(z))   e_{[d]} - \lambda   e_{[d]}\,]_{\sigma}.
        \end{align*}
        This means that
(\ref{eq:formulation1})-(\ref{eq:formulation1b}) is equivalent to 
     (\ref{eq:formulation2})-(\ref{eq:formulation2b}).
        Conversely, if this holds then  (\ref{eq:blockzero}) must hold because
        the Hamiltonian $H$ can be an arbitrary symmetric $\binom{n}{d} \times \binom{n}{d}$ matrix.

We shall now prove that (\ref{eq:blockzero}) holds
         if and only if $\sigma$ has the form $m [k]$.
The $(I,J)$ entry in   $\exp{(T(z))}$ is non-zero if and only if we can map from state $e_J$ to state $e_I$ via a composition of linear maps $X_{\alpha,\beta}$, 
where $|\alpha| = |\beta| \in \sigma$. This is possible
if and only if $J\backslash [d] \subset I \backslash [d]$ and
$$ \qquad        |\,I\backslash [d]\,| \,-\, |\,J \backslash [d]\,| \,\,=\,\, \sum_{k \in \sigma} p_k k \quad
\hbox{for some $p_k \in \mathbb{N}$.}$$

  Given any  $1 \leq j < i \leq d$, we can always find $I,J \subseteq [n]$ of levels $i$ and $j$ respectively such that $J\backslash [d] \subset I \backslash [d]$. 
  Hence the block $[\,\exp{(T(z))}\,]_{\sigma, \sigma^c} $ is non-zero if and and only if
        \begin{equation}\label{equivcond}
        \exists\, j \in [d]\backslash \sigma\,\, \,\exists\, i \in \sigma \,: \,\,i > j \,\,\, {\rm and} \,\,\,
            j = i - \sum_{k \in \sigma} p_k k
\quad \hbox{for some $p_k \in \mathbb{N}$.}
                    \end{equation}
        Let $m$ and $M$ be the minimal and maximal elements of $\sigma$. If $M - p_m m \in [d] \backslash \sigma$ for
        some $p_m \in \mathbb{N}_+$ then (\ref{equivcond}) holds
        and  $[\,\exp{(T(z))}\,]_{\sigma, \sigma^c}$ is non-zero. 
Suppose now that $M - p_m m \in \sigma$  for all positive integers $p_m < M/m$.
Then $M = km$ by the choice of $m$, and we have~$m[k] \subseteq \sigma$.

Next suppose  $m[k] = \sigma$. No $j \in [d] \backslash \sigma$  with $j < M$
         is  a multiple of $m$, and thus (\ref{equivcond})~fails.
        Finally suppose $m[k] \subsetneq \sigma$. We
        take the smallest element $i \in \sigma$ that is not a multiple of~$m$. Then $i - m \in [d]\backslash\sigma$ and (\ref{equivcond}) holds for $p_m = 1$. Therefore $[\,\exp{(T(z))}\,]_{\sigma, \sigma^c}$ is non-zero.
            \end{proof}

Theorem~\ref{thm:twoformulations} shows that the traditional formulation
(\ref{eq:formulation2b}) coincides with our formulation 
(\ref{eq:CC1}) in all cases that have
appeared in the computational chemistry literature
  \cite{FL, FO, KP1} including
CCS, CCD, CCSD, CCSDT.
In particular, for electronic structure Hamiltonians, our formulation~\eqref{eq:formulation1b} contains only expressions where CC amplitudes appear at most to the fourth power, because of the special structure of these Hamiltonians.
Scenarios where the two formulations differ, like $\sigma = \{2,3\}$,
are  less relevant for coupled cluster theory.
For us, (\ref{eq:CC1}) is more elegant 
than (\ref{eq:formulation2b}), and its algebraic degree is lower.
This is why we refer to (\ref{eq:CC1}) as {\bf the}  CC equations.
            
\section{Numerical Solutions}
\label{sec6}

This section covers the state-of-the-art for computing {\bf all} solutions to the CC equations. 
We use the formulation (\ref{eq:formulation1b})
as a square system with  $|\widetilde \sigma|+1$ unknowns. 
Readers familiar with earlier formulations of the CC equations may 
consult Theorem \ref{thm:twoformulations} for the precise relation.

Our new approach allows for the solution  of systems much larger than
those in~\cite{FO}. This is accomplished by
leveraging monodromy techniques. Throughout this section, we use Julia version 1.9.1 and {\tt HomotopyContinuation.jl} version 2.9.2~\cite{BT}. The python part is performed using Python 3.8.8 and PySCF 2.0.1~\cite{PyScf1}. 
Computations were done on the MPI-MiS computer server using four 18-Core Intel Xeon E7-8867 v4 at 2.4 GHz 
(3072 GB~RAM).

The beginning of our experiments, for given $d,n$,
 is the choice of a symmetric
matrix~$H$.
For a specific truncation level $\sigma$, we pick a pair $(\lambda, z)\in \mathbb{C} \times  \mathbb{C}^{|\widetilde \sigma|}$ at random.
We then construct a generic complex matrix $H$ for which $(\lambda, z)$ is a solution to (\ref{eq:formulation1b}).
This is done easily by solving a linear system of equations.
The reason is that (\ref{eq:formulation1b}) depends linearly on~$H$.

Having the starting data $H$ and ($\lambda$, $z$) we use the monodromy solver 
to find all solutions for $H$.
This system reveals the CC degree, and
is later used to compute quantum chemistry systems.
The degree of $V_\sigma$ is found in a similar manner, by slicing $V_\sigma$ with an appropriate generic linear space. Whenever feasible, we use {\tt Macaulay2} \cite{M2} to validate
 the degrees and CC degrees we found numerically.
In this manner we found
the table entries in Examples \ref{ex:dreisieben}, \ref{ex:vieracht} and \ref{ex:dreisechs3}.
Here is one more case:

\begin{example}[$d=3, n=8$]\label{ex:3e8o}
The CC systems for the six varieties for three electrons in eight spin-orbitals are
$$ \begin{matrix}
\sigma && \{1\} & \,\,\,\{2\} \,\,\,&\,\,\, \{3\} \,\,\,& \{1,2\} & \{1,3\} &\,\, \{2,3\}\,\, \\
|\widetilde \sigma| +1 && 16 &  31 & 11 & 46 & 26 & 41  \\
{\rm deg}(V_\sigma) && 6006 & 1 & 1 & 3894 & 4195 & 1 \\
{\rm CCdeg}_{3,8} && 38610 & 31 & 11 & 145608 & 58214 & 41 \\
{\rm \# real} &&   430 & 31 & 11 & 1376 & 658 & 41 \\
{\rm solve (sec)} && 619 & 8 & 3 & 26757 & 1948 & 7 \\
{\rm certify (sec)} && 7 & 3 & 0 & 41 & 8 & 0\\
\end{matrix}
$$
The number of real solutions (listed in row ``\#real'') varies for different choices of real-valued Hamiltonians, unless $V_\sigma$ is a linear space.
The counts $430,\, 1376,\, 658$ (in row ``\#real'') are  from {\it one} representative sample run for a random real-valued $H$.
The degree of $V_{\{1,2\}}$ is computed numerically. 
The runtimes in seconds are for solving and certifying for generic $H$.
\end{example}

\begin{example} \label{ex:complexity}
The  CC degrees found for various $d,n$ and $\sigma$
indicate the complexity of fully solving the CC equations.
Figure \ref{fig:Heat_plot} concerns $d$ electrons in $n=2d$ spin-orbitals.
We show  CC degrees for rank-complete levels of excitations, i.e. $\sigma = [k]$ for $k=1,\ldots , d$.
For $d=4$, $\sigma = [2]$ and $d = 5$, $\sigma = [1]$, the upper bound in Theorem \ref{thm:generalbound} is displayed.
The question marks indicate that the degree of $V_\sigma$ is unknown.
\begin{figure}[h!]
    \centering
    \includegraphics{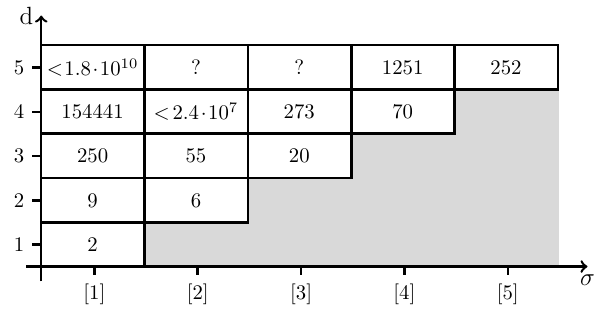}
    \caption{CC degrees for different truncation level with $n = 2d$.}
    \label{fig:Heat_plot}
\end{figure}
A striking observation in Figure~\ref{fig:Heat_plot} is the  low numbers on the diagonal.
This is because CC theory is exact in its untruncated limit, i.e.\ the number of solutions is the matrix size. 
For CC at level $[d-1]$, the  entries come from Proposition~\ref{prop:offby}.
The middle entries increase rapidly. The left column 
$[1]$ concerns the Grassmannian ${\rm Gr}(d,2d)$.
 \end{example}

Even for larger cases,
Theorem~\ref{thm:generalbound} gives
a good upper bound on the number of solutions.
The key ingredient is the degree of the truncation variety $V_{\sigma}$.
This is often easier to compute than ${\rm CCdeg}_{d,n}(\sigma)$.
For computing ${\rm deg}(V_\sigma)$, we used a range of techniques.
First of all, the degree is one in the linear cases of Theorem \ref{thm:linear}. 
Second, for the CCS truncation  ($\sigma = \{1\}$), the degree of the Grassmannian has an explicit description (cf.~\cite[Theorem 5.13]{MS}).
Third, sometimes we can compute ${\rm deg}(V_\sigma)$ symbolically with {\tt Macaulay2}~\cite{M2}; this requires either an explicit description of the ideal of the truncation variety, e.g., as provided in Theorem~\ref{thm:primeideal}, or can sometimes be done by implicitization  \cite[Section~4.2]{MS}
from the parametrization (\ref{eq:xtopsi}).
The degree of the considered variety can be computed by the degree of the ideal.
Finally, if this all fails, we use numerics,
taking advantage of the fact that $V_\sigma$ is a complete
intersection in affine space $\mathbb{C}^{\binom{n}{d} - 1}$, cut out by
the polynomials $x_I(\psi)$ in (\ref{eq:iotainverse}).
Namely, we intersect $V_\sigma$  with a generic affine-linear space of codimension 
$\binom{n}{d}-|\widetilde{\sigma}| = {\rm codim}(V_\sigma)$.
The number of points in the intersection is the
degree of $V_\sigma$. Using an appropriate parametrization of the
 affine linear-space we may use the monodromy solver in order to compute the degree.
 
 In conclusion, the inequality in Theorem~\ref{thm:generalbound} leads to 
upper bounds for the number of roots of the CC equations, even when
the equations are too large to be solved completely. It is instructive
to compare previously known bounds to those found by our new approach.

\begin{example}[Scaling of the number of roots] \label{ex:scaling}
For $d=2$ we consider the CC equations for singles
 ($\sigma = \{1\}$) and
 doubles ($\sigma = \{2\}$) investigating the scaling of the number of roots with respect to $n$.
Figure \ref{fig:scaling} shows different bounds in a log-lin plot.
The blue curve is the previous bound from \cite[Theorem 4.10]{FO} and the green curve is our new bound from Theorem~\ref{thm:generalbound}.
We moreover show the exact number of roots; for CCD (right panel) this is given in Corollary \ref{cor:usual} and for CCS (left panel), this is given in Theorem~\ref{conj:G2n}, here confirmed for $n\leq 10$.
The graphs show that the algebraic geometry in this paper leads to much improved bounds.
  
\begin{figure}[h!]
     \centering
     \begin{subfigure}[b]{0.45\textwidth}
         \centering
         \includegraphics[width=\textwidth]{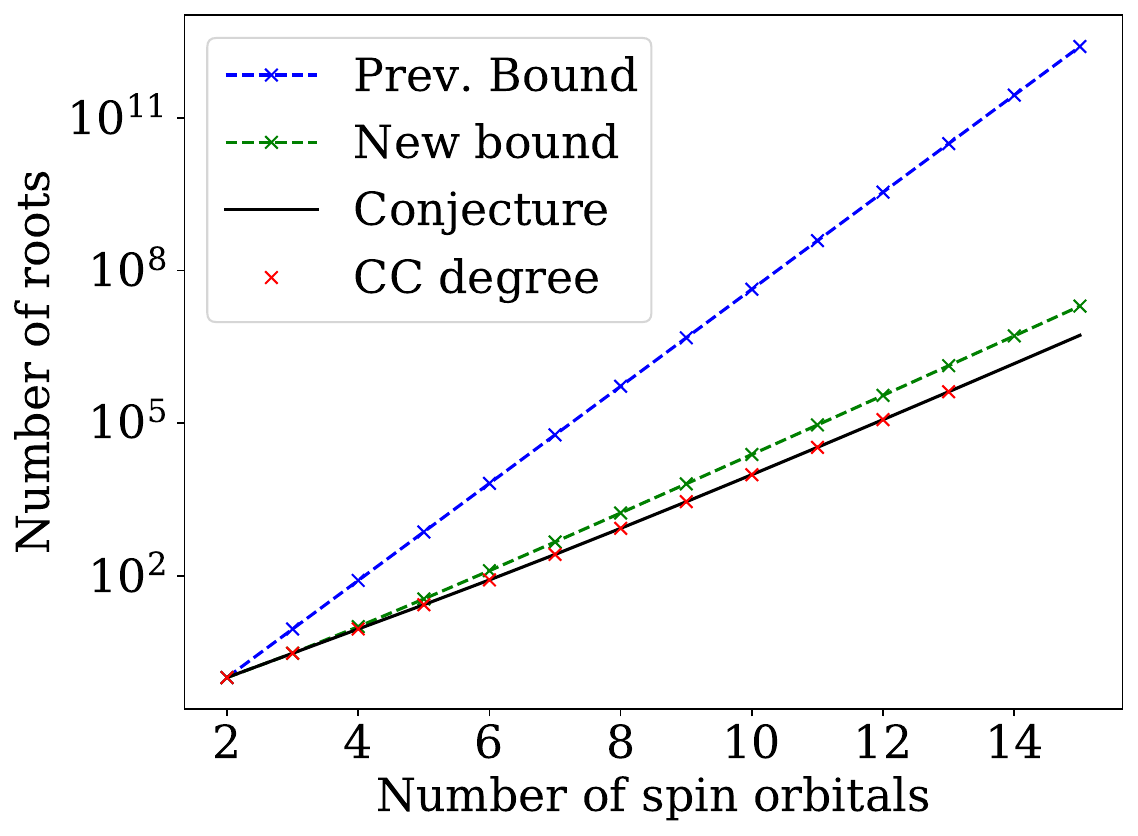}
         \label{fig:scaling_ccs}
     \end{subfigure}
     \hfill
     \begin{subfigure}[b]{0.45\textwidth}
         \centering
         \includegraphics[width=\textwidth]{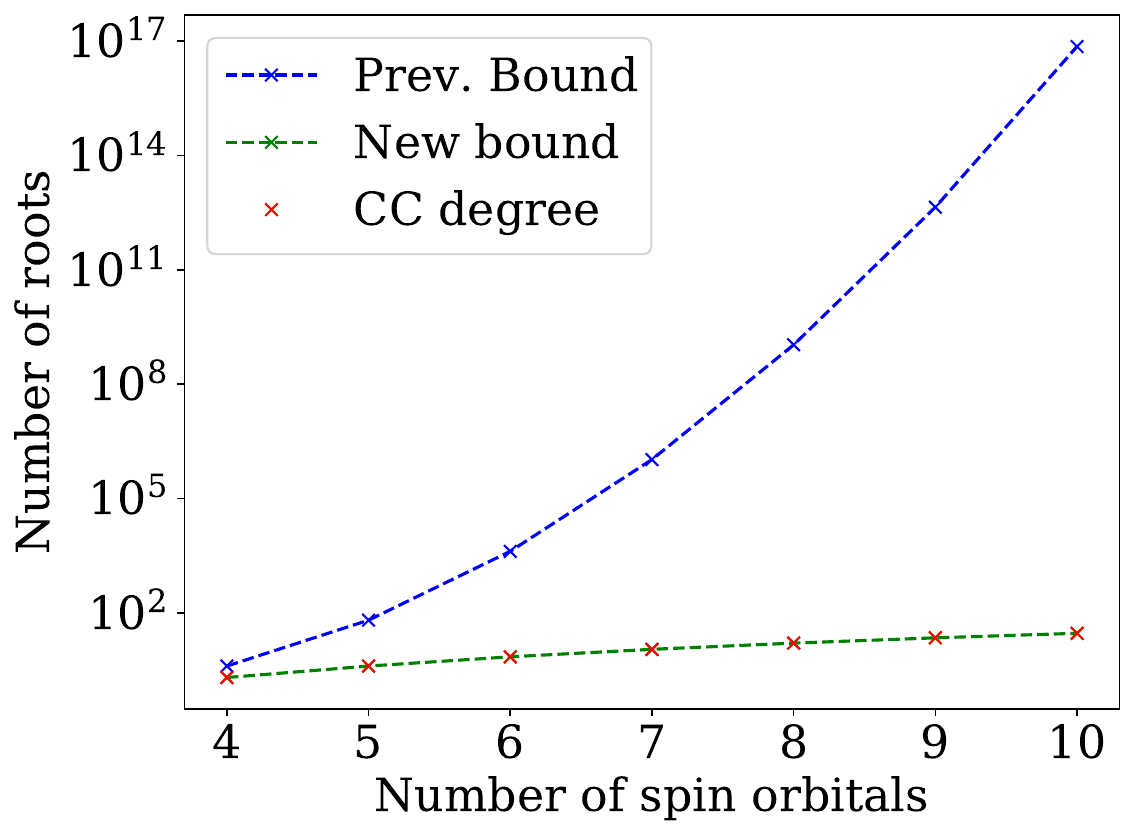}
         \label{fig:scaling_ccd} 
     \end{subfigure} \vspace{-0.2in}
     \caption{\label{fig:scaling}
     Bounds to the number of roots of CCS (left panel) and CCD (right panel).}
\end{figure}

\end{example}

Turning to solutions of the CC equations for real data, we use the parameter homotopy method to track paths from the generic starting solutions to solutions for a given quantum chemical Hamiltonian.
The number of these paths is the CC degree.
If this tracking converges, we find either a {\em non-singular solution} or a  {\em singular solution}.
A solution is singular if the Jacobian matrix of the polynomial system is singular (i.e.\ non-invertible), or if the winding number of the solution path is greater than one.
A singular solution could indicate that the solution variety has an extraneous component of dimension $\geq 1$. 
We observe that this is common in applications from quantum chemistry, such as those in Examples \ref{ex:ending48} and~\ref{ex:ending38}. 
However, inspecting the eigenvalues only, it appears that singular solutions still yield a good approximation see Example~\ref{ex:ending48}.
For a general Hamiltonian $H$, all paths converge to non-singular solutions,
and the number of solutions to (\ref{eq:formulation1b}) is exactly 
the CC degree.
Note that the Hamiltonian $H$ arising in quantum chemistry (cf.~Section \ref{sec4}) is not generic,
but has special structure. 
Therefore, the obtained number of solutions for the target system 
can be much smaller than the CC~degree.

\begin{figure}[h!]
     \centering
     \begin{subfigure}[b]{0.495\textwidth}
    \centering
    \includegraphics[width = \textwidth]{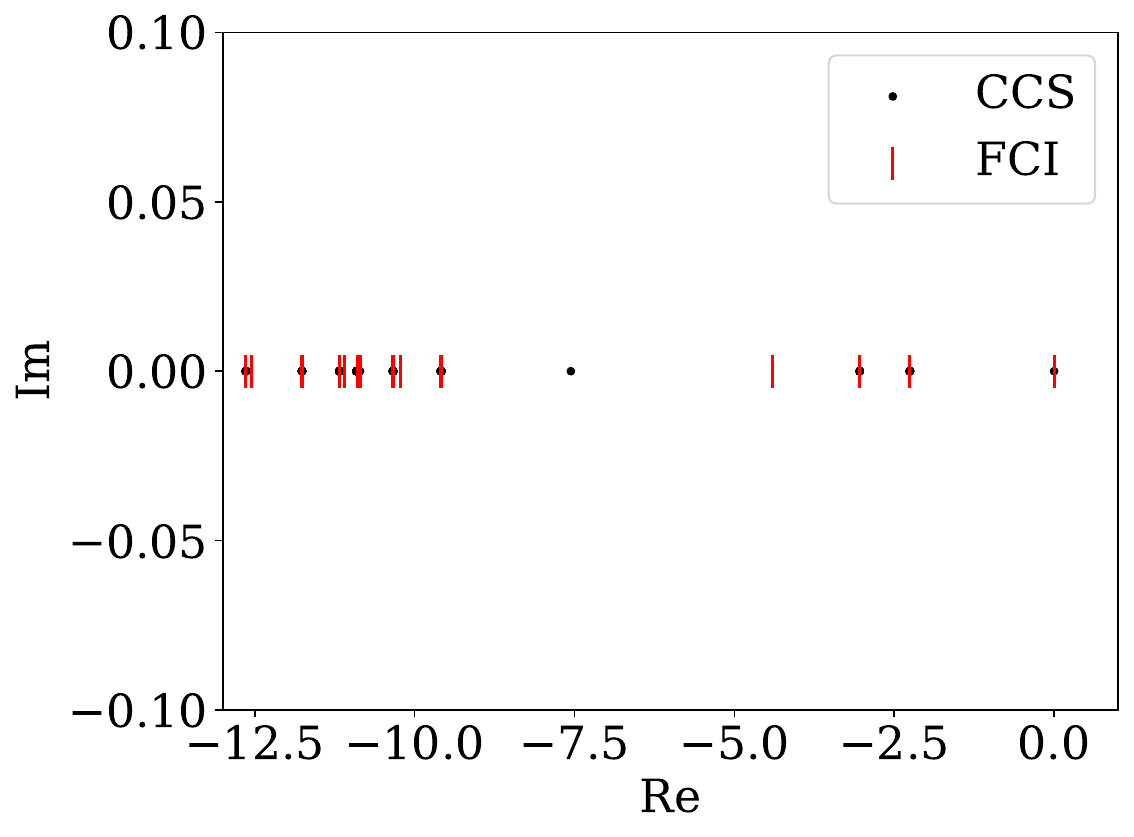}
    \caption{CCS (only real)}
    \label{fig:CCD_energies}
     \end{subfigure}
     \hfill
     \begin{subfigure}[b]{0.495\textwidth}
    \centering
    \includegraphics[width = \textwidth]{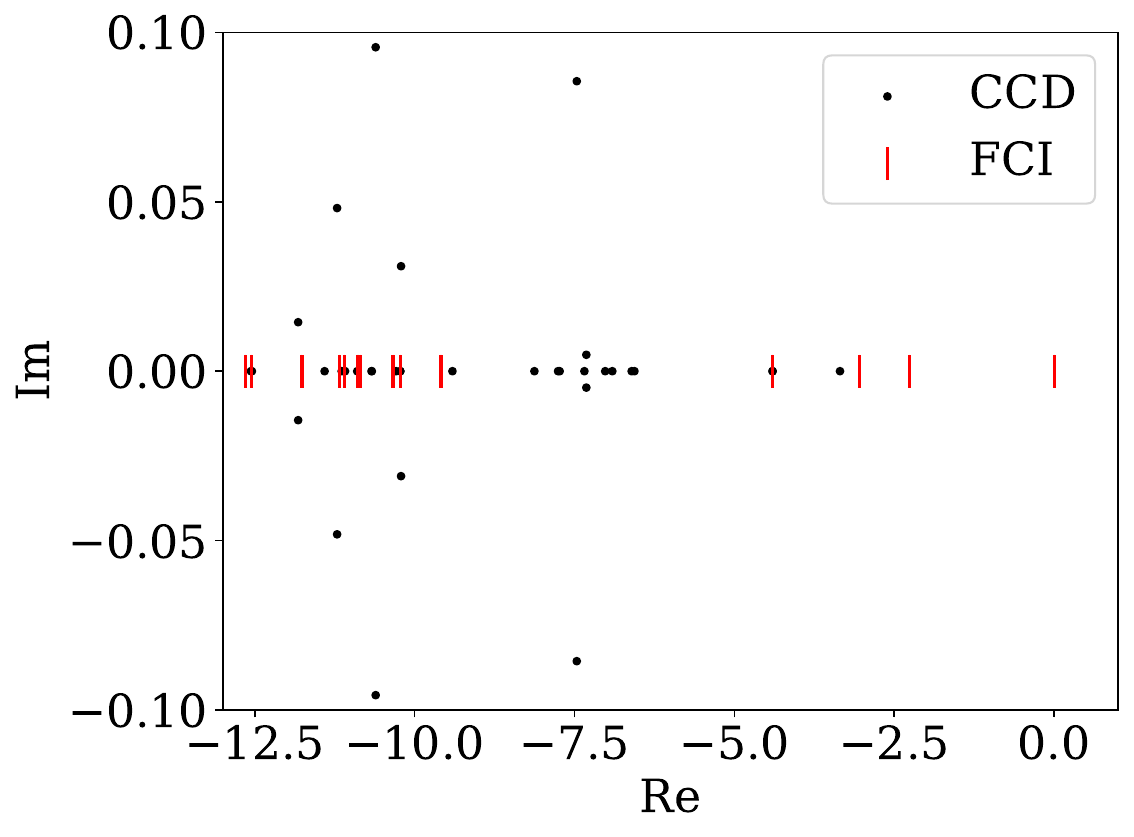}
    \caption{CCD}
    \label{fig:CCD_energies}
     \end{subfigure}
     \caption{\label{fig:LiH_energies}The real eigenvalues from  exact diagonalization (FCI),
     shown as red bars,
     are compared to the energy spectra, shown in black,
      from CCS and CCD (in the STO-6G basis).
     }
\end{figure}

\begin{example}[Lithium hydride ($d=4$, $n=8$)] \label{ex:ending48}
We use the Hamiltonian from Example~\ref{ex:morerunning}.
The computation of the generic start system for $\sigma = \{1\}$ takes 82 minutes and yields ${\rm CCdeg}_{4,8}(\{1\}) = 154441$, as predicted by Example \ref{ex:vieracht}.
Tracking all paths yields $3$ non-singular solutions all of which are real. 
We also find $104641$ singular solutions. Only $399$ of them yield real energies.
 We use these for the comparison to the exact eigenvalues. 
These calculations take $11$ minutes and $32$ seconds. 
For $\sigma = \{2\}$, the computation of the generic start system takes $13$ seconds and yields ${\rm CCdeg}_{4,8}(\{2\}) = 73$. 
Tracking all paths yields $36$ non-singular solutions of which $24$ are real and zero singular solutions. 
This takes less than one second. 
In Figure~\ref{fig:LiH_energies} we compare the exact eigenvalue spectrum with the energies obtained from CCS and CCD. 
An interesting observation is that CCS and CCD appear to approximate different subsets of eigenvalues that together cover the whole  spectrum.
\end{example}

Since the CC degree for $d = 4$, $n = 8$ and $\sigma =\{1,2\}$ is  currently unknown, we cannot apply our method to CCSD for lithium hydride yet.
In order to investigate the approximation quality for CCSD, we lower the number of electrons by looking at the lithium atom. 

\begin{example}[Lithium ($d\!=\!3$, $n\!=\!8$, $\sigma \!=\! \{1,2\}$)] \label{ex:ending38}
We find one non-singular solution, which is real, and $2931$ singular solutions. 
Compare this to
${\rm CCdeg}_{3,8}(\{1,2\}) = 145608$.
The runtime for the parameter homotopy was about one hour.
We compare the eigenvalue spectrum  of 
the $56 \times 56$ matrix $H$ with the energies obtained from the CCSD computations in Figure~\ref{fig:LiCCSD}.
    
\begin{figure}[h!]
     \centering
     \begin{subfigure}[b]{0.495\textwidth}
         \centering
         \includegraphics[width=\textwidth]{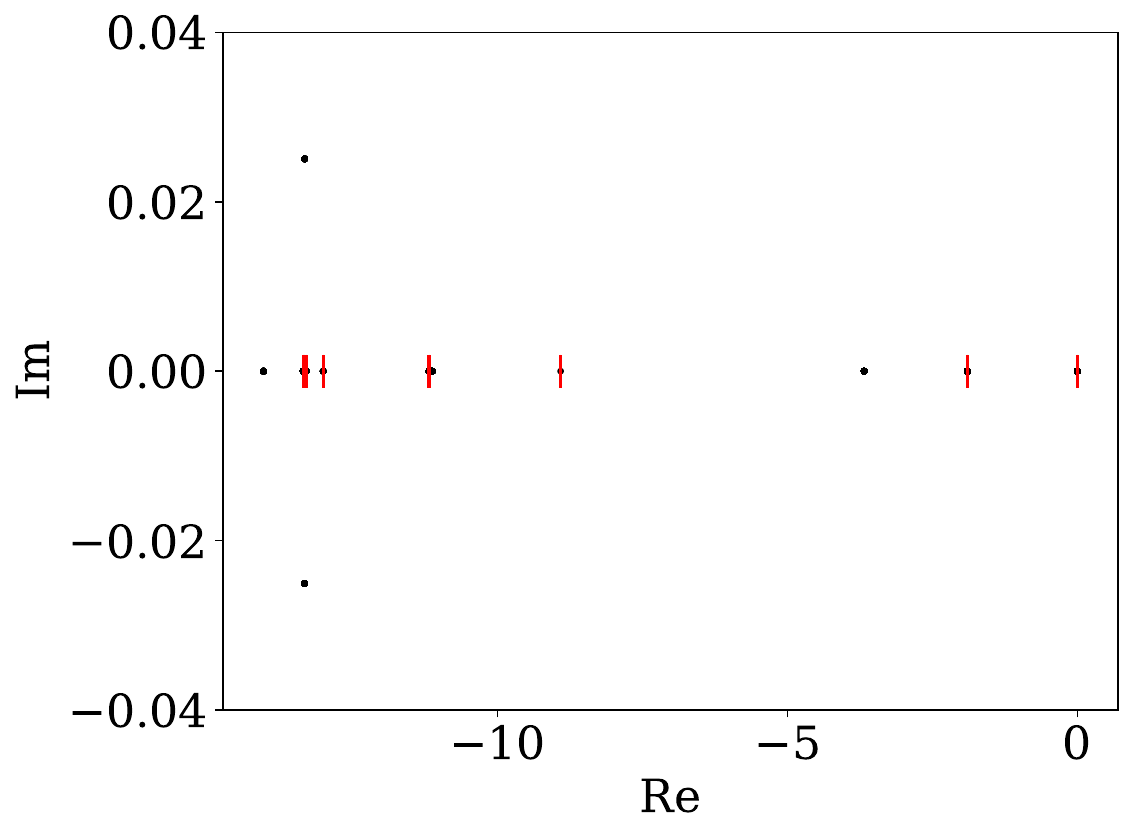}
         \label{fig:LiCCSD_all}
     \end{subfigure}
     \hfill
     \begin{subfigure}[b]{0.495\textwidth}
         \centering
         \includegraphics[width=\textwidth]{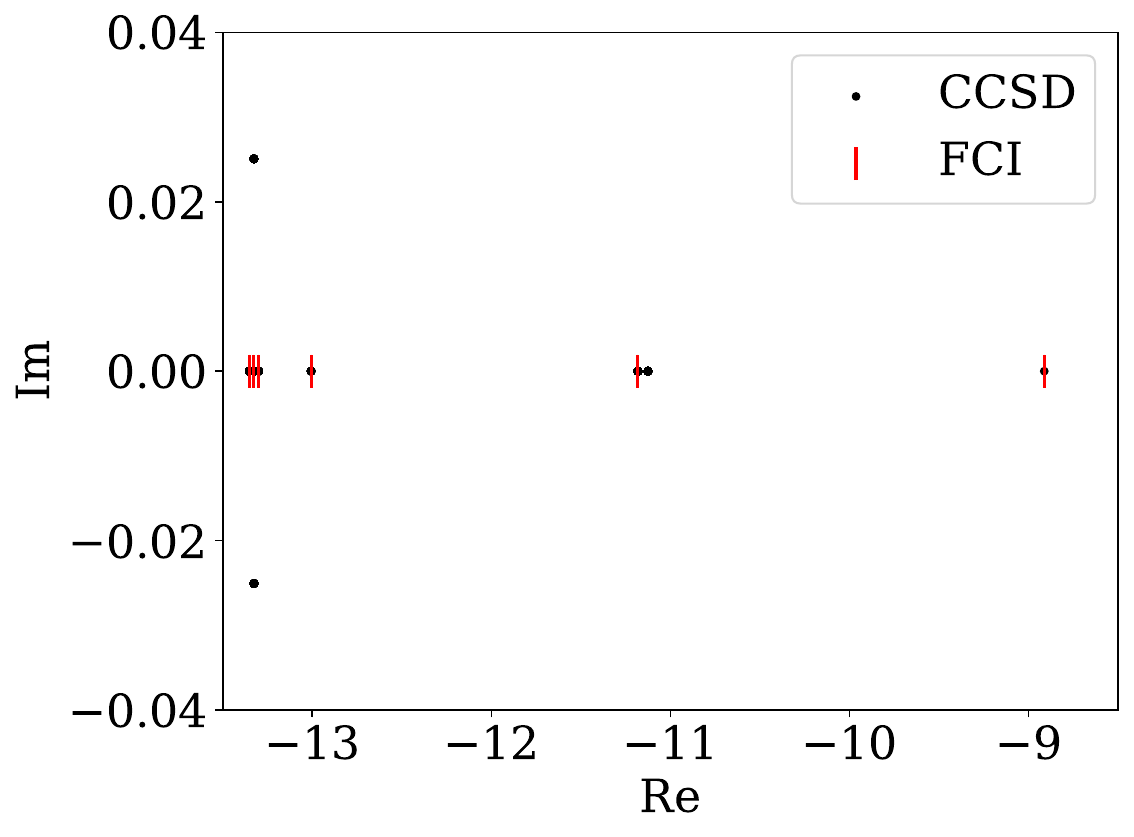}
         \label{fig:LiCCSD_zoom}
     \end{subfigure}
     \vspace{-0.45in}
     \caption{\label{fig:LiCCSD} Energy spectra from exact diagonalization (FCI) and from CCSD (in the STO-6G basis). The left panel shows the full spectrum. The right panel shows the spectral values between $-13.5$ and~$-8.5$. }
\end{figure}
\end{example}

For special Hamiltonians $H$, the
number of isolated solutions to the CC equations can be
much lower than the CC degree. This happens in applications, as seen in
Examples \ref{ex:ending48} and \ref{ex:ending38}.
It would be interesting to understand this phenomenon
for $H$ on special loci in the space of symmetric matrices.  
The next example suggests such a study for low rank matrices.

\begin{example}[$d=2,n=4, \sigma = \{1\}$]
The CC degree is $9$ for the Pl\"ucker quadric ${\rm Gr}(2,4)$.
Let $H$ be a general symmetric $6 \times 6$ matrix of rank $r$.
For $r = 1,2,3,4,5$, the number of solutions to (\ref{eq:formulation1b}) is
$1,3,5,7,9$.  We see this with Cramer's Rule as in proof of Proposition~\ref{prop:offby}.
\end{example}

In conclusion, this article introduced a new formulation
of the coupled cluster (CC) equations
in electronic structure theory. This rests on
a novel class of  algebraic varieties,
called truncation varieties. They live in the same projective space as
the Grassmannian, which they generalize.
Section \ref{sec6} has 
demonstrated that our new approach
leads to significant advances in numerically computing
{\bf all} roots of the CC equations.
Current off-the-shelf
software can now reliably solve instances of
actual interest in quantum chemistry. These practical advances
rest on the theorems in Sections \ref{sec2}, \ref{sec3} and \ref{sec5}.
We believe that those are of interest in their own right.
Many open problems and new avenues of inquiry were presented.
One concrete task is to find a formula for 
$ {\rm CCdeg}_{d,n}(\{1\})$, which is the
CC degree of the Grassmannian ${\rm Gr}(d,n)$.
We hope that this topic will 
interest experts in intersection theory.

\bigskip
 \bigskip

\noindent
\footnotesize
{\bf Authors' addresses:}

\smallskip

\noindent Fabian Faulstich,
Rensselaer Polytechnic Institute
\hfill {\tt faulsf@rpi.edu}

\noindent Bernd Sturmfels,
MPI-MiS Leipzig
\hfill {\tt bernd@mis.mpg.de}

\noindent Svala Sverrisd\'ottir,
UC Berkeley
\hfill {\tt svalasverris@berkeley.edu}

\end{document}